\documentclass[12pt]{amsart}
\usepackage{graphicx,color,enumerate,amssymb}
\usepackage{amsfonts,amsthm,amsmath,amssymb,tabularx,epsf}%booktabs,
\usepackage{listings}
\usepackage{url}
\usepackage{pgf,tikz,pgfplots}
\pgfplotsset{compat=1.15}
\usepackage{mathrsfs}
\usetikzlibrary{arrows}
\usepackage{graphics}
\usepackage[latin1]{inputenc}
\usepackage[T1]{fontenc}
\usepackage[english]{babel}
\usepackage{float}
\theoremstyle{plain}% Theorem-like structures provided by amsthm.sty
\newtheorem{theorem}{Theorem}[section]
\newtheorem{lemma}[theorem]{Lemma}

\newtheorem{corollary}[theorem]{Corollary}

\theoremstyle{definition}
\newtheorem{definition}[theorem]{Definition}
\newtheorem{example}[theorem]{Example}

\theoremstyle{remark}
\newtheorem{remark}{Remark}

\usepackage[body={21.6cm,27.9cm},top=3cm,left=3cm,right=3cm,bottom=3cm]{geometry}
\usepackage{amsfonts}
\usepackage{amssymb}
\usepackage{amsmath}
\usepackage{tikz}
   \usetikzlibrary{positioning,arrows,calc}
   \tikzset{
   modal/.style={>=stealth,shorten >=1pt,shorten <=1pt,auto,node distance=1.5cm,
   semithick},
   world/.style={circle,draw,minimum size=0.5cm,fill=gray!15},
   point/.style={circle,draw,inner sep=0.5mm,fill=black},
   reflexive above/.style={->,loop,looseness=7,in=120,out=60},
   reflexive below/.style={->,loop,looseness=7,in=240,out=300},
   reflexive left/.style={->,loop,looseness=7,in=150,out=210},
   reflexive right/.style={->,loop,looseness=7,in=30,out=330}
   }
\usetikzlibrary{decorations.markings}
\definecolor{cadmiumgreen}{rgb}{0.0, 0.42, 0.24}

\begin{document}
\pagestyle{myheadings}

\title[Limit points of $A_{\alpha}$-matrices of graphs]{Limit points of $A_{\alpha}$-matrices of graphs}
\subjclass{05C50, 05C05, 15A18}
\keywords{spectral radius; tree;  limit points;}
%%%%%%%%%%%%%%%%%%%%%%%%%%%%%%%%%%%%%%%%%%%%%%%%%%%%%%% nome

\author[E. R. Oliveira]{Elismar R. Oliveira}
\address{Instituto de Matem\'atica e Estat\'{\i}stica, UFRGS, Porto Alegre, Brazil}
\email{\tt elismar.oliveira@ufrgs.br}

\author[V. Trevisan]{Vilmar Trevisan}
\address{Instituto de Matem\'atica e Estat\'{\i}stica, UFRGS, Porto Alegre, Brazil }
\email{\tt trevisan@mat.ufrgs.br}

\begin{abstract}
We study limit points of the spectral radii of $A_{\alpha}$-matrices of graphs. Adapting a method used by J. B. Shearer in 1989, we prove a density property of $A_{\alpha}$-limit points of caterpillars for $\alpha$ close to zero. Precisely, we show that for $\alpha \in [0, 1/2)$ there exists a positive number $\tau_2(\alpha)>2$ such that any value $\lambda> \tau_2(\alpha)$ is an $A_{\alpha}$-limit point. We also determine the existence of other intervals for which all its points are $A_{\alpha}$-limit points.
\end{abstract}

\maketitle
%%%%%%%%%%%%%%%%%%%%%%%%%%%%%%%%%%%%%%%%%%%%%%%%%%%%%%%%%%%%%%%%%%%%%%%%%%%%%%%%%%%%%%%%%%%%%%%%%%%%%%%%%%%%%%%%%%%%%%%%%%%%%%%%%%%%%%%%%%%%%

\section{Introduction}

The 1972 seminal paper of A. J. Hoffman~\cite{hoffman1972limit} introduced the concept of limit points of eigenvalues of graphs. Let $\mathcal{A}$ be the set of all symmetric matrices of all orders, in which every entry is a non-negative integer, and $R=\{\rho:\rho=\rho(A) \mbox{ for some } A\in\mathcal{A}\}$ where $\rho(A)$ is the largest eigenvalue of $A$. Hoffman asked which real numbers can be limit points of $R$ and showed that it is sufficient to consider matrices of $\mathcal{A}$ having entries in $\{0, 1\}$ and 0 diagonal, e.g. adjacency matrices of graphs. Additionally, he determined all limit points of $R \leq\sqrt{2+\sqrt{5}}$.

In 1989,  a remarkable result due to J. B. Shearer~\cite{shearer1989distribution} extended the work of Hoffman. He showed that every real number larger than $\sqrt{2+\sqrt{5}}$ is a limit point of $R$.  There is a considerable amount of literature originated from this seminal paper, extending the results to other matrices related to graphs, as well as to eigenvalues other than the spectral radius. We refer to the  survey paper \cite{wang2020hoffman} for an account of the many nice results.

In this paper, we are interested in the Hoffman's original question, which deals only with limit points of the spectral radius of graphs. We recall from \cite{nikiforov2017Aalpha} that, for an undirected graph $G$,  the matrix $A_{\alpha}(G):= \alpha D(G)+ (1-\alpha)A(G)$, for $0\leq \alpha\leq 1$. Inspired by the quote of \cite{wang2020hoffman}: \emph{``Therefore, it is reasonable to focus efforts to obtain a formula for the $A_{\alpha}$-counterpart of the $A$-limit point $\sqrt{2+\sqrt{5}}$''}, we study the $A_{\alpha}$ version of Hoffman and Shearer's results.

We translate and generalize Shearer's proof from \cite{shearer1989distribution} using techniques of eigenvalue location. Our main result (Theorem~\ref{thm: main alpha limit one}) is that for any $\alpha \in [0, 1/2)$ there exists a positive number $\tau_2(\alpha)>2$ such that any value $\lambda> \tau_2(\alpha)$ is an $A_{\alpha}$-limit point. Additionally, we study (Theorem~\ref{thm: main alpha limit two}), for small values of $\alpha$, the existence of intervals $[\tau_1(\alpha),\tau_1'(\alpha))$ for which all numbers are also $A_{\alpha}$-limit points.

The paper is organized as follows. In the next Section~\ref{sec:basic}, we set the notation, give the necessary preliminaries and explain the strategy, including the main tool - eigenvalue location, that is used to obtain our results. In Section~\ref{sec:shearer}, we use eigenvalue location to adapt Shearer's method to the $A_{\alpha}$ matrix. In this section, for a given $\lambda >2$, we define a sequence of graphs $G_k$ whose spectral radius is convergent to some (unknown) value $\lambda^*$. Section~\ref{sec:convergence} contains the main result of this note. We develop several convergence criteria so the constructed graph $G_k$ have its spectral radius converging to  $\lambda$. Our main result, Theorem~\ref{thm: main alpha limit one} generalizes Shearer's result for the adjacency matrix to any  $A_{\alpha}$ matrix: For $0 \leq \alpha <1/2$, we find a number $\tau_2(\alpha)$ so that if $\lambda \in [\tau_2(\alpha), +\infty)$ then $\lambda$ is a limit point for the spectral radius of $A_\alpha$. In Section~\ref{sec:small}, we study when $\lambda < \tau_2(\alpha)$. We determine intervals $I$ for which all numbers in $I$ are limit points for $A_\alpha$, for $0\leq \alpha \leq \alpha^{*}:=\frac{3-\sqrt{2}}{7}= 0.226540919+$.

\section{Basic ideas and main tool}\label{sec:basic}
For an undirected graph $G$, let $D(G)$ be its diagonal degree matrix. From \cite{nikiforov2017Aalpha} and \cite{wang2020hoffman}, we introduce, for $0\leq \alpha\leq 1$, the matrix $$A_{\alpha}(G):=A(G)+ \alpha (D(G)-A(G)).$$

It is easy to see that $A_{0}(G)=A(G)$, the adjacency matrix of $G$ and $A_{1}(G)=D(G)$ the degree matrix of $G$. Also, $A_{1/2}(G)=1/2(D(G)+A(G))=1/2 Q(G)$ where $Q$ is the signless Laplacian matrix of $G$. The largest eigenvalue of $A_\alpha(G)$ is the spectral radius and denote by $\rho_\alpha(A_\alpha(G))$

From \cite{WangLiuBel2020} and \cite{wang2020hoffman} we know that for $0\leq \alpha\leq 1$, if $\displaystyle\lim_ {k \to \infty}\rho_{\alpha}(A_{\alpha}(G_{k}))= \lambda$ then $\lambda \geq 2$.

\begin{lemma}[\cite{nikiforov2017Aalpha}, \cite{wang2020hoffman}] \label{lem: propr spectral A alpha}
For every connected graph $G$ with maximum vertex degree $\Delta$, and for every $\alpha \in[0,1]$, the $A_\alpha$-spectral radius $\rho_{A_\alpha}(G)$ satisfies the following properties:\\
(i) $\frac{1}{2}\left(\alpha(\Delta+1)+\sqrt{\alpha^2(\Delta+1)^2+4 \Delta(1-2 \alpha)}\right) \leqslant \rho_{A_\alpha}(G) \leqslant \Delta$;\\
(ii) if $H$ is a proper subgraph of $G$, then $\rho_{A_\alpha}(H)<\rho_{A_\alpha}(G)$;\\
(iii) if $0 \leqslant \alpha<\beta \leqslant 1$, then $\rho_{A_\alpha}(G)<\rho_{A_\beta}(G)$.
\end{lemma}

Intuitively, we argue that, since the entries of $A_{\alpha}(G)$ depend continuously  on $\alpha$, the limit behaviour of its limit points should resembles the one verified for the adjacency matrix ($A_{0}(G)$) itself.  By formalizing this reasoning, we aim to characterize, for $\alpha$ small enough, intervals consisting of $A_{\alpha}$-limit points.

\subsection{Strategy}\label{sec:strategy}

Our strategy is to prove that the Shearer's argument is stable for $\alpha$ close to $\alpha=0$ because $A_{0}(G)=A(G)$ is the adjacency matrix of $G$, which has good properties and all the limit points are known.

We recall that in \cite{shearer1989distribution} it  was proved that for each $\lambda  \geq \sqrt{2+\sqrt{5}}$ there exists a sequence of caterpillar graphs $G_k$ such that $\displaystyle\lim_ {k \to \infty} \rho(G_k) = \lambda$. So we will look at the distribution of the spectral radius of the $A_{\alpha}$-matrices of this kind of trees.

This will be done by eigenvalue location techniques. We refer to  \cite{BragaRodrigues2017} or \cite{HopJacTrevBook22} for details on applications of the \emph{Algorithm Diagonalize}:\\
\begin{figure}[H]
{\tt
\begin{tabbing}
aaa\=aaa\=aaa\=aaa\=aaa\=aaa\=aaa\=aaa\= \kill
     \> Input: a symmetric matrix $M = (m_{ij})$ with underlying tree $T$\\
     \> Input: a bottom up ordering $v_1,\ldots,v_n$ of $V(T)$\\
     \> Input: a real number $x$ \\
     \> Output: a diagonal matrix $D = \mbox{diag}(d_1, \ldots, d_n)$ congruent to $M + xI$ \\
     \> \\
     \>  Algorithm $\mbox{Diag}(M, x)$ \\
     \> \> initialize $d_i := m_{ii} + x$, for all $i$ \\
     \> \> {\bf for } $k = 1$ to $n$ \\
     \> \> \> {\bf if} $v_k$ is a leaf {\bf then} continue \\
     \> \> \> {\bf else if} $d_c \neq 0$ for all children $c$ of $v_k$ {\bf then} \\
     \> \> \>  \>   $d_k := d_k - \sum \frac{(m_{ck})^2}{d_c}$, summing over all children of $v_k$ \\
     \> \> \> {\bf else } \\
     \> \> \> \> select one child $v_j$ of $v_k$ for which $d_j = 0$  \\
     \> \> \> \> $d_k  := -\frac{(m_{jk})^2}{2}$ \\
     \> \> \> \> $d_j  :=  2$ \\
     \> \> \> \> if $v_k$ has a parent $v_\ell$, remove the edge $\{v_k,v_\ell\}$. \\
     \> \>  {\bf end loop} \\
\end{tabbing}
}
\caption{\label{fig:diagonalize} Diagonalizing $M + xI$ for a symmetric
matrix $M$ associated with the tree $T$.}
\end{figure}
The following theorem summarizes the way in which the algorithm will be applied, we refer to~\cite{HopJacTrevBook22} for details.

\begin{theorem} (Sylvester's Law of Inertia, see \cite[Theorem 4.5.8]{HorJohnson85}),\label{thm: inertia}
Let $M$ be a symmetric matrix of order $n$ that corresponds to a weighted tree $T$ and let $x$ be a real number. Given a bottom-up ordering of $T$, let $D$ be the diagonal matrix produced by Algorithm Diagonalize with entries $T$ and $x$. The following hold:
\begin{itemize}
    \item[(a)] The number of positive entries in the diagonal of $D$ is the number of eigenvalues of $M$ (including multiplicities) that are greater than $-x$.
    \item[(b)] The number of negative entries in the diagonal of $D$ is the number of eigenvalues of $M$ (including multiplicities) that are less than $-x$.
    \item[(c)] The number of zeros in the diagonal of $D$ is the multiplicity of $-x$ as en eigenvalue of $M$.
\end{itemize}
\end{theorem}

In our case, we consider $M=A_{\alpha}$ and $x=-\lambda$, where $\lambda$ is the spectral radius of $M$, which in general is not known, obtaining $\operatorname{Diag}(A_{\alpha}, -\lambda)$. We notice that, as $A_{\alpha}(G):=A(G)+ \alpha (D(G)-A(G))= (1-\alpha)A(G)+ \alpha D(G)$, then by construction we initialize the tree with $1-\alpha$ in each edge and $\alpha d(v) -\lambda$ in each vertex $v$.

As an example, which we will use later, we compute the $A_{\alpha}$-limit point for some special starlike trees $T_{1,n,n}$ (one path of length 1 and two paths of length $n$ connected to a root). In \cite{OliveTrevAppDiff} the authors proved that when $n \to \infty$ the spectral radius of the adjacency matrix of $T_{1,n,n}$ converges to $\sqrt{2+\sqrt{5}}=2.058+$.

\begin{theorem}\label{thm: alpha stalike limit} Let $T_{1,n,n}$ be the above described starlike tree, for $n \geq 1$. If $0\leq \alpha \leq 1$ and $n \to \infty$ then \\ $\displaystyle\lim_{k\to \infty} \rho_{\alpha}(A_{\alpha}(T_{1,n,n})) = \tau_{0}(\alpha) \in \left[\sqrt{2+\sqrt{5}} , 3\right]$, $\tau_{0}(\alpha)$ is the only solution of \\ \hspace*{3cm} $3\alpha -\lambda - \frac{(1-\alpha)^2}{\alpha -\lambda}-2\theta'_{\alpha}=0$,\\ where $\theta'_{\alpha}=\frac{(2\alpha-\lambda) +\sqrt{(2\alpha -\lambda)^2  - 4(1-\alpha)^2}}{2}$,  for $\lambda >2$.\\
Moreover, the correspondence $\alpha \mapsto \tau_{0}(\alpha)$ is strictly increasing.
\end{theorem}
\begin{proof}
We notice that $A_{\alpha}(T_{1,n,n}):=A(T_{1,n,n})+ \alpha (D(T_{1,n,n})-A(T_{1,n,n}))= (1-\alpha)A(T_{1,n,n})+ \alpha D(T_{1,n,n})$, so by construction we initialize the tree with $1-\alpha$ in each edge, $3\alpha -\lambda$  at the root, $2\alpha -\lambda$ in each internal vertex of the paths and $\alpha -\lambda$ in each leaf, as illustrated in Figure~\ref{fig:initializeT1nn}.
\begin{figure}[H]
  \centering
  \includegraphics[width=12cm]{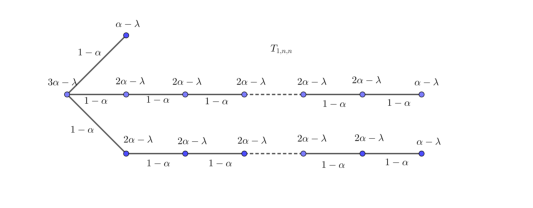}
  \caption{Initializing $\operatorname{Diag}(A_{\alpha}(T_{1,n,n}), -\lambda)$}\label{fig:initializeT1nn}
\end{figure}

After processing the tree, we obtain
\begin{figure}[H]
  \centering
  \includegraphics[width=9cm]{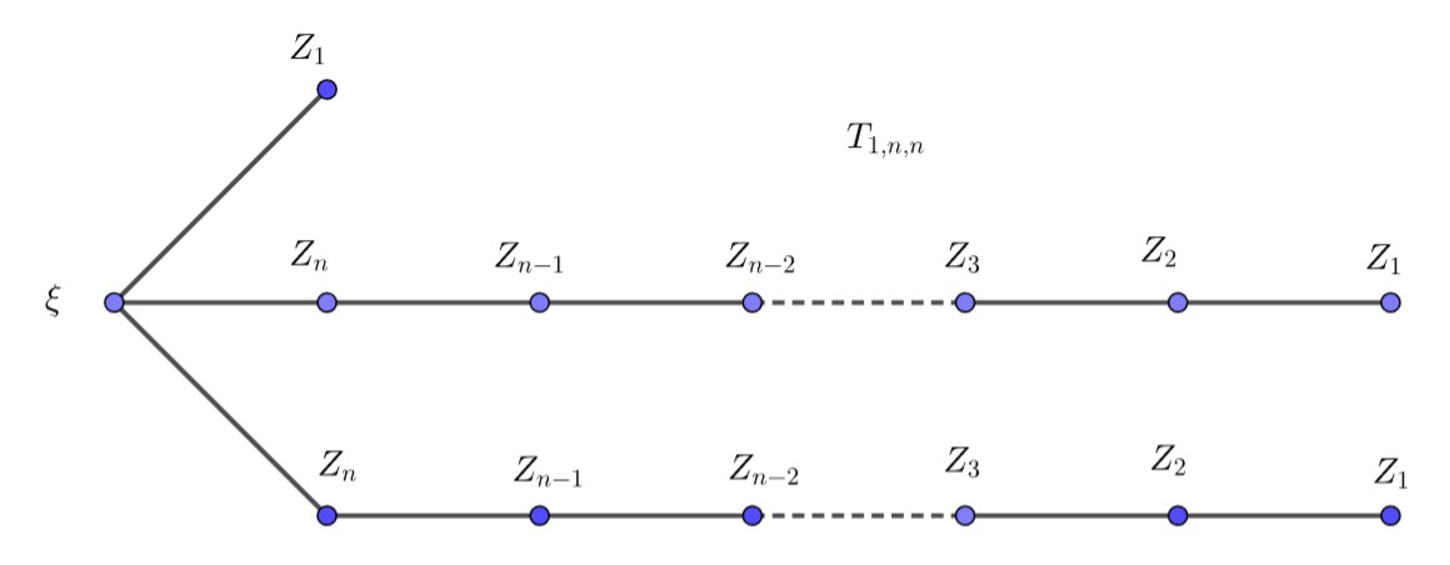}
  \caption{Result of $\operatorname{Diag}(A_{\alpha}(T_{1,n,n}), -\lambda)=(Z_1, \ldots, Z_n, \xi)$}\label{fig:diagonalizeT1nn}
\end{figure}
In Figure~\ref{fig:diagonalizeT1nn} we depicted the output of $\operatorname{Diag}(A_{\alpha}(T_{1,n,n}), -\lambda)$, which we now describe. In each leaf we get $Z_1:=\alpha -\lambda$, and through the paths $P_n$ we obtain $Z_{j+1}:= \varphi_{\alpha}(Z_{j}),\; j\geq 1$ where $\varphi_{\alpha}(t) = 2\alpha -\lambda - \frac{(1-\alpha)^2}{t}, \; t \neq 0$. Finally, at the root, we have
$$\xi:=3\alpha -\lambda - \frac{(1-\alpha)^2}{Z_1}-2 \frac{(1-\alpha)^2}{Z_n}=0.$$

Rational recurrences as $Z_{j+1}:= \varphi_{\alpha}(Z_{j}),\; j\geq 1$ were studied in \cite{OliveTrevAppDiff}. Their behaviour is quite simple and depend on the quantity
$$\Delta_{\alpha}:=(2\alpha -\lambda)^2 + 4(-(1-\alpha)^2)= (2\alpha -\lambda)^2  - 4(1-\alpha)^2.$$
Since $\lambda > 2$, we have
$$-\lambda < -2,$$
%$$2\alpha-\lambda < 2\alpha-2< 0$$
%$$(2\alpha-\lambda)^2 >  4(\alpha-1)^2$$
$$(2\alpha-\lambda)^2 - 4(1-\alpha)^2 >0.$$

As $\Delta_{\alpha} >0$, we have two distinct fixed points (i.e. $\varphi_{\alpha}(t) = t$),
$$\theta_{\alpha}:=\frac{(2\alpha-\lambda) - \sqrt{(2\alpha -\lambda)^2  - 4(1-\alpha)^2}}{2}$$
and $\theta'_{\alpha}:=\frac{(2\alpha-\lambda) +\sqrt{(2\alpha -\lambda)^2  - 4(1-\alpha)^2}}{2}$ (note that $\theta_{\alpha} \, \theta'_{\alpha}= (1-\alpha)^2 $,  $\theta'_{\alpha} - \theta_{\alpha}= ^{\sqrt{(2\alpha -\lambda)^2  - 4(1-\alpha)^2}} $  and $\theta'_{\alpha} + \theta_{\alpha}= 2\alpha-\lambda $ ). In this case $Z_n \to \theta_{\alpha}$ when $n \to \infty$.

In this way, taking the limit (it exists because we have a sequence of sub-graphs and the maximum degree is bounded, see Lemma~\ref{lem: propr spectral A alpha}) when $n \to \infty$ in $\xi$, we obtain an $A_{\alpha}$-limit point $\lambda$ which satisfies the equation:
$$3\alpha -\lambda - \frac{(1-\alpha)^2}{\alpha -\lambda}-2 \frac{(1-\alpha)^2}{\theta_{\alpha}}=0\text{
or }3\alpha -\lambda - \frac{(1-\alpha)^2}{\alpha -\lambda}-2\theta'_{\alpha}=0,$$
because $\theta_{\alpha} \, \theta'_{\alpha}= (1-\alpha)^2$.

For future purpose, we now define the following function $F_{0}:(2, +\infty)\times [0,1] \to \mathbb{R}$ given by
\begin{equation}\label{eq:F0}
F_{0}(\lambda, \alpha):=3\alpha -\lambda - \frac{(1-\alpha)^2}{\alpha -\lambda}-2\theta'_{\alpha} = \alpha - \frac{(1-\alpha)^2}{\alpha -\lambda} - \sqrt{(2\alpha -\lambda)^2  - 4(1-\alpha)^2}.
\end{equation}
By definition, when $\alpha \to 0$ both $\alpha -\lambda$ and $(2\alpha -\lambda)^2  - 4(1-\alpha)^2$ are respectively non zero and positive, so $F_{0}$ is continuous as algebraic combination of simple functions.

By elementary calculus techniques, one can see that
$\displaystyle\lim_ {\lambda \to 2}F_{0}(\lambda, \alpha) = \frac{1}{2-\alpha}>0$, $\displaystyle\lim_ {\lambda \to \infty}F_{0}(\lambda, \alpha) = -\infty < 0$ and $\frac{\partial F_{0}(\lambda, \alpha)}{\partial \lambda } = -{\frac { \left( 1-\alpha \right) ^{2}}{ \left( \alpha-\lambda
 \right) ^{2}}}+ \,{\frac {2\,\alpha - \,\lambda}{\sqrt { \left( 2\,
\alpha-\lambda \right) ^{2}-4\, \left( 1-\alpha \right) ^{2}}}} <0$
meaning that there exists a unique solution $\lambda$ of $F_{0}(\lambda, \alpha)=0$.

This equation($F_{0}(\lambda, \alpha)=0$) implicitly defines $\lambda$ as a function of $\alpha$. Let $\tau_{0}(\alpha)$ be such function. A tedious computation shows that $\tau_{0}(\alpha):=\lambda$ is a positive root of a polynomial $P_{\alpha}(\lambda)=0$, where $P_\alpha(\lambda)= -{\lambda}^{4}+6\,\alpha\,{\lambda}^{3}+ \left( -8\,{\alpha}^{2}-8\,
\alpha+4 \right) {\lambda}^{2}+ \left( 4\,{\alpha}^{3}+12\,{\alpha}^{2
}-6\,\alpha \right) \lambda-8\,{\alpha}^{3}+8\,{\alpha}^{2}-4\,\alpha+
1$. Moreover, $F_{0}(\lambda, \alpha)<0$ for $\lambda> \tau_{0}(\alpha)$.

For example, for $\alpha=0$ we have $P_{0}(\lambda)= -\lambda^{4}+4\lambda^{2} +1$ so  $\tau_{0}(0)=\sqrt{2+\sqrt{5}}=2.058171027$, as expected from \cite{shearer1989distribution}. Also, when $\alpha \to 1^{-}$ we get $F_{0}(\lambda, 1)=\displaystyle\lim_ {\alpha \to 1}F_{0}(\lambda, \alpha) =  -\sqrt { \left( \lambda-2 \right) ^{2}}+1$ so $F_{0}(\lambda, 1)=0$ only if $\lambda=3$, thus $\tau_{0}(1) =3$. A careful computation shows that $\frac{\partial F_{0}(\lambda, \alpha)}{\partial \alpha} = \left( {\frac { \left( 1-\alpha \right) ^{2}}{ \left( \alpha-\lambda
 \right) ^{2}}}+1 \right) ^{2}+{\frac {2\,\lambda-4}{\sqrt { \left( 2
\,\alpha-\lambda \right) ^{2}-4\, \left( 1-\alpha \right) ^{2}}}}>0$ so $\tau_{0}$ is an increasing function of $\alpha$ (at each root $\lambda=\tau_{0}(\alpha)$  the value of $F_0$ increases with $\alpha$ because $\frac{\partial F_{0}(\lambda, \alpha)}{\partial \alpha}>0$, so the next root is necessarily grater than $\lambda$) as exemplified in the Table~\ref{tab:tau0}.\\
\end{proof}
\begin{table}
  \centering
   \begin{tabular}{|c|r|}
   \hline
   $\alpha$ & $\tau_{0}(\alpha)$\\
   \hline
    0 & 2.058171027  \\
    $10^{-5}$ & 2.058172154  \\
    $10^{-4}$ & 2.058182294  \\
    $10^{-3}$ & 2.058283826  \\
    $10^{-2}$ & 2.059312583  \\
    $10^{-1}$ & 2.071110742  \\
    $0.3$    & 2.111760279  \\
    $0.5$ & 2.191487884\\
    $0.9$ & 2.727297451\\
    $0.9999$ & 2.999700025\\
   \hline
 \end{tabular}  \caption{Values for $\tau_0(\alpha)$.}\label{tab:tau0}
\end{table}

It is worth noticing that in \cite{shan2024ordering} the authors have shown that the $A_\alpha$-spectral radius increases according to the shortlex ordering of the length of its pendant paths. This kind of construction provides a sequence of $A_\alpha$-limit points close to 2 and hopefully the smallest value for an interval of $A_\alpha$-limit points, at least when $\alpha \sim 0$. We will investigate that in the next sections.

\section{Shearer's approach}\label{sec:shearer}
In this section, we follow  Shearer's approach for $A_{\alpha}$-limit points. As a matter of fact, we reinterpret Shearer's work in terms of eigenvalue location for the $A_{\alpha}$ matrix of \textit{caterpillars}. For a given $\lambda > 2$, Shearer constructs a sequence of graphs $G_k(\lambda)$ such that $\rho(G_k(\lambda))\rightarrow \lambda $. For a path $P_k$ with vertices $v_1, \ldots, v_k$, let us denote by $G_{k}(\lambda)=[r_1, r_2, \ldots, r_k]$ the graph, called caterpillar, having  $r_i$ pendant vertices at each vertex, for $i=1,\ldots,k$. We refer to Figure \ref{fig:caterpillar} for an illustration.
\begin{figure}[H]
  \centering
  \includegraphics[width=12cm]{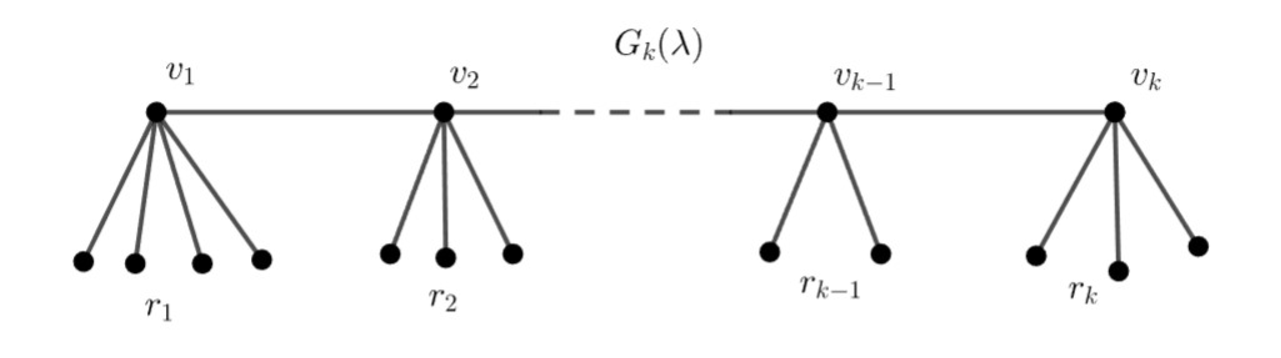}
  \caption{Caterpillar $G_{k}(\lambda)$}\label{fig:caterpillar}
\end{figure}
 We now explain how the technique of eigenvalue location can be used to adapt Shearer's method (for adjacency matrix) to the $A_\alpha$ matrix. We want to have  $2 < \rho(G_k)< \lambda $ (for any  $k$) and the crucial point is to determine the number $r_{k+1}$ of pendant vertices for $G_{k+1}$, keeping this property. We compute $\mbox{Diag}(G_{k}(\lambda), -\lambda)=(b_1, \ldots, b_k)$,  where $b_j$ are the values at the vertices $v_j$, $j=1,\ldots, k$. Because $\lambda$ is larger than the largest eigenvalue of $G_k(\lambda)$, we require  that all the $b_j's$ are negative, by Theorem~\ref{thm: inertia}. At the leaves, we have $\alpha-\lambda<0$ (so there is no need to consider these values)  and in each vertex $v_i$, the number $b_i$ is as follows:\\
$$b_1:=\alpha(r_1 +1) -\lambda -r_1  \frac{(1-\alpha)^2}{\alpha -\lambda}= \alpha -\lambda + r_1 \left(\alpha+ \frac{(1-\alpha)^2}{\lambda-\alpha} \right),$$
$$b_2:=\alpha(r_2 +2) -\lambda - \frac{(1-\alpha)^2}{b_{1}} -r_2  \frac{(1-\alpha)^2}{\alpha -\lambda}= 2\alpha -\lambda - \frac{(1-\alpha)^2}{b_{1}} + r_2 \left(\alpha+ \frac{(1-\alpha)^2}{\lambda-\alpha} \right)$$
and so on, resulting in the recurrence
\[
\left\{
  \begin{array}{ll}
    b_1=\alpha -\lambda + r_1 \left(\alpha+ \frac{(1-\alpha)^2}{\lambda-\alpha} \right)=\alpha -\lambda + r_1 \delta_{\alpha} \\
    b_{j+1}= 2\alpha -\lambda - \frac{(1-\alpha)^2}{b_{j}} + r_{j+1} \left(\alpha+ \frac{(1-\alpha)^2}{\lambda-\alpha} \right)=\varphi_{\alpha}(b_{j})+ r_{j+1} \delta_{\alpha}, 1 \leq j \leq k-1\\
    b_{k}= \alpha -\lambda - \frac{(1-\alpha)^2}{b_{k-1}} + r_{k} \left(\alpha+ \frac{(1-\alpha)^2}{\lambda-\alpha} \right)=-\alpha+ \varphi_{\alpha}(b_{k-1})+ r_{k} \delta_{\alpha}
  \end{array}
\right.
\]
where $\varphi_{\alpha}(t) = 2\alpha -\lambda - \frac{(1-\alpha)^2}{t}, \; t \neq 0$ and $\delta_{\alpha}:=\alpha+ \frac{(1-\alpha)^2}{\lambda-\alpha}$.

We notice that this new sequence $b_j$ differs from $Z_j$ by a multiple $r_j$ of a positive  drift $\delta_{\alpha}=\alpha+ \frac{(1-\alpha)^2}{\lambda-\alpha} \to \frac{1}{\lambda}$ when $\alpha \to 0$.

In order to have $\rho_{\alpha}(A_{\alpha}(G_k)) < \lambda$, we require that $b_{j}<0$ for all $j$ (see Theorem~\ref{thm: inertia}). From \cite{OliveTrevAppDiff} we know that it  is actually necessary to require $b_{j}<\theta'_{\alpha}$ for all $j$. This happens because $\theta'_{\alpha}<0$ is a repelling fixed point (from the right side) so if any iterate satisfy $b_{j}>\theta'_{\alpha}$ then some future $b_{j+m}$ will be positive.

From this, observing that in each step we should make $r_j$ as big as possible (in other words $\theta'_{\alpha} - \delta_{\alpha} < b_{j} < \theta'_{\alpha}$), we obtain the following formula for the $r_j$'s:
\[
\left\{
  \begin{array}{ll}
    \alpha -\lambda + r_1 \delta_{\alpha}<\theta'_{\alpha} \\
    \varphi_{\alpha}(b_{j})+ r_{j+1} \delta_{\alpha}<\theta'_{\alpha}, 1 \leq j \leq k-1\\
    -\alpha+ \varphi_{\alpha}(b_{k-1})+ r_{k} \delta_{\alpha}<\theta'_{\alpha}
  \end{array}
\right.
\]
or
\begin{equation}\label{eq: shearer alpha seq}
\left\{
  \begin{array}{ll}
     r_1=\left\lfloor\frac{1}{\delta_{\alpha}}\left(\theta'_{\alpha} -(\alpha -\lambda)\right)\right\rfloor \\
    r_{j+1}=\left\lfloor\frac{1}{\delta_{\alpha}}\left(\theta'_{\alpha} - \varphi_{\alpha}(b_{j})\right)\right\rfloor, 1 \leq j \leq k-1\\
r_{k}=\left\lfloor\frac{1}{\delta_{\alpha}}\left(\theta'_{\alpha} - \varphi_{\alpha}(b_{k-1})+\alpha \right)\right\rfloor
  \end{array}
\right.
\end{equation}

\begin{definition}\label{def: shearer alpha seq}
   For a given $\lambda >2$, the sequence of caterpillars  $G_k$ satisfying  equation~\eqref{eq: shearer alpha seq} is  the \textit{$\alpha$-Shearer} sequence associated with $\lambda$.
\end{definition}

By construction, the graph $G_k$ satisfies the inequality $\rho_{\alpha}(A_{\alpha}(G_k)) < \lambda$. Also, $G_k$ is always a subgraph of $G_{k+1}$ (the next $r_{k+1} \geq 0$ so we can use it as a pendant path). From Lemma~\ref{lem: propr spectral A alpha},  $\rho_{\alpha}(A_{\alpha}(G_k)) < \rho_{\alpha}(A_{\alpha}(G_{k+1}))$, thus there exists the limit
$$\displaystyle\lim_ {k\to \infty} \rho_{\alpha}(A_{\alpha}(G_k)) \leq \lambda.$$
It remains unclear weather  the equality occurs or not! This is the subject of the next section.

\section{A convergence criterion}\label{sec:convergence}
It is hard to derive a general argument for the convergence from \cite{shearer1989distribution}, however the algorithm $\operatorname{Diag}(G_{k}(\lambda), -\lambda)$ provides a quick answer to that  question. As $\rho_{\alpha}(A_{\alpha}(G_k))$ is an increasing sequence, it will converges to $\lambda$ if, and only if, for any $\varepsilon>0$ there exists $k$ such that  $\rho_{\alpha}(A_{\alpha}(G_k))> \lambda - \varepsilon$, in other words, if some output of $\operatorname{Diag}(G_{k}(\lambda), -(\lambda - \varepsilon))=(b_1(\varepsilon), \ldots, b_k(\varepsilon))$ is positive (every output $ b_k(\varepsilon)$ is a function of $\varepsilon$ defined at some neighborhood of $0$).

From the above consideration we deduce the following criterion.
%\vt{I think criteria is plural}
\begin{theorem}\label{thm: criteria one} Let $G_k$ be the $\alpha$-Shearer sequence associated with $\lambda>2$ and $\varepsilon_k$ the sequence consisting of the smaller positive root of the function $\varepsilon \mapsto b_{k}(\varepsilon)$, for each $k$. Then,
$ \displaystyle\lim_{k\to \infty} \rho_{\alpha}(A_{\alpha}(G_k)) = \lambda,$
if and only if
$ \displaystyle \lim_{k\to \infty} \varepsilon_{k}=0.$
\end{theorem}
\begin{proof}
We translate the above criterion in terms of each output $b_j(\varepsilon)$ of $\operatorname{Diag}(G_{k}(\lambda), -(\lambda - \varepsilon))$. In each case, the function $\varepsilon \mapsto b_j(\varepsilon)$ will be studied.

First we consider $\varepsilon \mapsto b_1(\varepsilon)=\alpha -(\lambda - \varepsilon) + r_1 \delta_{\alpha}(\varepsilon)$, where
$$\delta_{\alpha}(\varepsilon):= \alpha+ \frac{(1-\alpha)^2}{\lambda - \varepsilon -\alpha}.$$

Notice that $b_1(0)=b_1 < 0$ and $b_1(\varepsilon)$ is differentiable for $0<  \varepsilon< \lambda -\alpha$ with
$$\frac{d\, b_1}{d\, \varepsilon}=1+ r_{1}\frac{(1-\alpha)^2}{(\lambda - \varepsilon -\alpha)^2}>0$$
and $\displaystyle\lim_ {\varepsilon\to (\lambda -\alpha)^{-}} b_1(\varepsilon)=+\infty$ meaning that there exists a unique root $\varepsilon_{1} \in (0, \lambda -\alpha) $ of $b_1(\varepsilon)$. In conclusion, $b_1(\varepsilon)<0$ for all $\varepsilon \in (0, \varepsilon_{1})$ and positive if $\varepsilon> \varepsilon_{1}$. So, if we want improve our approximation beyond $\varepsilon_{1}$ we need to look into $b_2(\varepsilon)$ and higher indices.

Now, just to fix the ideas, we consider $\varepsilon \mapsto b_2(\varepsilon)= 2\alpha -(\lambda - \varepsilon) - \frac{(1-\alpha)^2}{b_{1}} + r_{2} \delta_{\alpha}(\varepsilon)$.

Notice that $b_2(0)=b_2 < 0$ and $b_2(\varepsilon)$ is differentiable for $0<  \varepsilon< \varepsilon_{1}< \lambda -\alpha$ with
$$\frac{d\, b_2}{d\, \varepsilon}=1+  \frac{(1-\alpha)^2}{b_{1}^2}\frac{d\, b_1}{d\, \varepsilon} +r_{2} \frac{(1-\alpha)^2}{(\lambda - \varepsilon -\alpha)^2}>0$$
and $\displaystyle\lim_ {\varepsilon\to \varepsilon_{1}^{-}} b_2(\varepsilon)=+\infty$ meaning that there exists a unique root $\varepsilon_{2} \in (0, \varepsilon_{1}) $ of $b_2(\varepsilon)$.  In conclusion, $b_2(\varepsilon)<0$ for all $\varepsilon \in (0, \varepsilon_{2})$ and positive if $\varepsilon > \varepsilon_{2}$.

Proceeding in this way, one can see that our better approximation using $G_k$ is $\varepsilon_{k-1}< \cdots < \varepsilon_{1}$, except for the last output $b_k$. We recall that the function
$\varepsilon \mapsto b_k(\varepsilon)= \alpha -(\lambda - \varepsilon) - \frac{(1-\alpha)^2}{b_{k-1}} + r_{k} \delta_{\alpha}(\varepsilon)$.

Notice that $b_k(0)=b_k < 0$ and $b_k(\varepsilon)$ is differentiable for $0<  \varepsilon< \varepsilon_{k-1}$ with
$$\frac{d\, b_k}{d\, \varepsilon}=1+  \frac{(1-\alpha)^2}{b_{k-1}^2}\frac{d\, b_{k-1}}{d\, \varepsilon} + r_{k} \frac{(1-\alpha)^2}{(\lambda - \varepsilon -\alpha)^2}>0$$
and $\displaystyle\lim_ {\varepsilon\to \varepsilon_{k-1}^{-}} b_k(\varepsilon)=+\infty$ meaning that there exists an unique root $\varepsilon_{k} \in (0, \varepsilon_{k-1}) $ of $b_k(\varepsilon)$.  In conclusion, $b_1(\varepsilon), \ldots, b_k(\varepsilon)<0$ for all $\varepsilon \in (0, \varepsilon_{k})$ and  $b_k(\varepsilon)$ is positive if $\varepsilon > \varepsilon_{k}$.

Now that we have proved that $\varepsilon_{k}$ is a well defined decreasing sequence, in order to complete our proof suppose that $\displaystyle\lim_{k\to \infty} \rho_{\alpha}(A_{\alpha}(G_k)) = \lambda,$ then for any $\varepsilon> 0$ there exists $k$ such that $\rho_{\alpha}(A_{\alpha}(G_k)) > \lambda - \varepsilon$. Since  $\varepsilon_{k}< \varepsilon$ we get $ \displaystyle \lim_{k\to \infty} \varepsilon_{k}=0.$\\
Reciprocally, if $ \displaystyle \lim_{k\to \infty} \varepsilon_{k}=0$, it means that for any $\varepsilon>0$ we can find $k$ such that $b_k(\varepsilon)$ is positive, that is, $\rho_{\alpha}(A_{\alpha}(G_k)) > \lambda - \varepsilon_{k}$. As $ \displaystyle \lim_{k\to \infty} \varepsilon_{k}=0$ we get $\displaystyle\lim_{k\to \infty} \rho_{\alpha}(A_{\alpha}(G_k)) = \lambda$.
\end{proof}

This criterion is fine, but it has only theoretical use, since it requires the computation of all the sequence $\varepsilon_{k}$, so it will be more useful to have a  sufficient computable condition.

A more suitable criterion would be the following result.
\begin{theorem}\label{thm: criteria two} Let $G_k$ be the $\alpha$-Shearer sequence associated with $\lambda>2$ and $\sigma_k= \frac{-b_k}{\frac{d\, b_k}{d\, \varepsilon}(0)}$ the sequence defined by the root of the line tangent to the function $\varepsilon \mapsto b_{k}(\varepsilon)$, at $\varepsilon=0$, for each $k$. Then,
$\displaystyle\lim_{k\to \infty} \rho_{\alpha}(A_{\alpha}(G_k)) = \lambda,$
if and only if
$\displaystyle \lim_{k\to \infty} \sigma_{k}=0.$
\end{theorem}
\begin{proof}
  In order to the result, we take a closer look at the geometrical construction of the $\varepsilon_{j}$'s.
\begin{figure}[H]
  \centering
  \includegraphics[width=12cm]{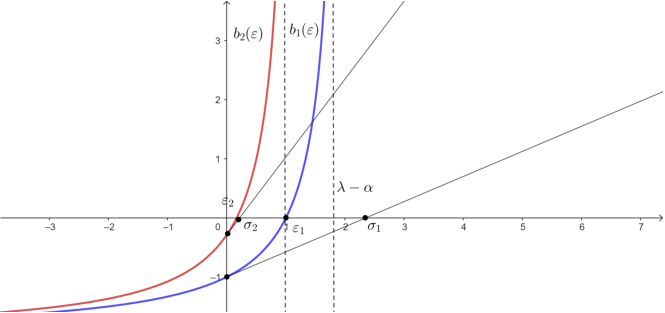}
  \caption{Sequence $\sigma_{k}$}\label{fig:criteria}
\end{figure}

We can see (an easy computation shows that $\frac{d^{2}b_j}{d^{2}\varepsilon}>0$) that each function $b_j(\varepsilon)$ is concave in the interval $(0, \varepsilon_{k})$ thus we can find (see Figure~\ref{fig:criteria}) a sequence of points $\varepsilon_{j+1}< \sigma_{j+1} < \varepsilon_{j}$ (assuming $\varepsilon_{0}= \lambda -\alpha$), so that  $\displaystyle\lim_ {k\to \infty} \varepsilon_{k}=0$ if, and only if, $\displaystyle\lim_ {k\to \infty} \sigma_{k}=0$.

More specifically, $\sigma_{k}$ is the only root of the tangent line of $b_k(\varepsilon)$, by $\varepsilon=0$, that is,
$$0= b_k(0)+ \frac{d\, b_k}{d\, \varepsilon}(0) (\sigma_{k} -0 )$$
or
$$\sigma_{k} = \frac{-b_k}{\frac{d\, b_k}{d\, \varepsilon}(0)}.$$
\end{proof}

We now seek a sufficient condition for the convergence of the limit point of the $\alpha$-Shearer sequence.
\begin{theorem}\label{thm: criteria three} Let $G_k$ be the $\alpha$-Shearer sequence associated with $\lambda>2$. If the numerical sequence
$$\frac{(1-\alpha)^2}{b_{k-1}^2}  +  \frac{(1-\alpha)^2}{b_{k-1}^2}\frac{(1-\alpha)^2}{b_{k-2}^2} + \cdots +  \frac{(1-\alpha)^2}{b_{k-1}^2}\cdots\frac{(1-\alpha)^2}{b_{1}^2}$$
diverges, when $k \to \infty$, then
$$\displaystyle\lim_ {k\to \infty} \rho_{\alpha}(A_{\alpha}(G_k)) = \lambda.$$
\end{theorem}
\begin{proof}
We recall that
$$\frac{d\, b_k}{d\, \varepsilon}=1+  \frac{(1-\alpha)^2}{b_{k-1}^2}\frac{d\, b_{k-1}}{d\, \varepsilon} + r_{k} \frac{(1-\alpha)^2}{(\lambda - \varepsilon -\alpha)^2}>1+  \frac{(1-\alpha)^2}{b_{k-1}^2}\frac{d\, b_{k-1}}{d\, \varepsilon}$$
and repeating this evaluation we get
$$\frac{d\, b_k}{d\, \varepsilon}> 1+  \frac{(1-\alpha)^2}{b_{k-1}^2}\left(1+  \frac{(1-\alpha)^2}{b_{k-2}^2}\frac{d\, b_{k-2}}{d\, \varepsilon}\right)$$
$$= 1+  \frac{(1-\alpha)^2}{b_{k-1}^2}  +  \frac{(1-\alpha)^2}{b_{k-1}^2}\frac{(1-\alpha)^2}{b_{k-2}^2}\frac{d\, b_{k-2}}{d\, \varepsilon} $$
and finally
$$\frac{d\, b_k}{d\, \varepsilon}>  1+  \frac{(1-\alpha)^2}{b_{k-1}^2}  +  \frac{(1-\alpha)^2}{b_{k-1}^2}\frac{(1-\alpha)^2}{b_{k-2}^2} + \cdots +  \frac{(1-\alpha)^2}{b_{k-1}^2}\cdots\frac{(1-\alpha)^2}{b_{1}^2}.$$

When $\frac{(1-\alpha)^2}{b_{k-1}^2}  +  \frac{(1-\alpha)^2}{b_{k-1}^2}\frac{(1-\alpha)^2}{b_{k-2}^2} + \cdots +  \frac{(1-\alpha)^2}{b_{k-1}^2}\cdots\frac{(1-\alpha)^2}{b_{1}^2} \to \infty$ we also get  $\frac{d\, b_k}{d\, \varepsilon}(0) \to \infty$ so that $\sigma_{k} = \frac{-b_k}{\frac{d\, b_k}{d\, \varepsilon}(0)}\to 0$  because $\theta'_{\alpha} -\delta_{\alpha} < b_k < \theta'_{\alpha}$.
\end{proof}

The next two statements are the main results of this paper and are consequences of Theorem~\ref{thm: criteria three}.

\begin{theorem}\label{thm: main alpha limit one} Let $\alpha <1/2$, and $\delta_{\alpha}=\alpha+ \frac{(1-\alpha)^2}{\lambda-\alpha}$.
Then $\lambda$ is an $A_{\alpha}$-limit point for any $\lambda\geq \tau_{2}(\alpha)$, where $\tau_{2}(\alpha)$ is the only solution of $-1+\alpha= \theta'_{\alpha} - \delta_{\alpha}$.
\end{theorem}

\begin{proof}
Let $G_k$ be the $\alpha$-Shearer sequence associated with $\lambda>2$. We want to show that $\displaystyle\lim_ {k\to \infty} \rho_{\alpha}(A_{\alpha}(G_k)) = \lambda$. In order to use Theorem~\ref{thm: criteria three} it is sufficient to show that
$b_{j}^2 < (1-\alpha)^2$  because it is equivalent to $ 1 < \frac{(1-\alpha)^2}{b_{j}^2}$, which causes the  divergence of the summation $\frac{(1-\alpha)^2}{b_{k-1}^2}  +  \frac{(1-\alpha)^2}{b_{k-1}^2}\frac{(1-\alpha)^2}{b_{k-2}^2} + \cdots +  \frac{(1-\alpha)^2}{b_{k-1}^2}\cdots\frac{(1-\alpha)^2}{b_{1}^2}$.

Since $b_{j}<0$ we obtain the equivalent condition $-1+\alpha< b_{j}$. By definition,
$$\theta'_{\alpha} - \delta_{\alpha} < b_{j} < \theta'_{\alpha}$$
so it is enough to show that $-1+\alpha< \theta'_{\alpha} - \delta_{\alpha},$ or explicitly
$$ -1 +2\alpha + \frac{(1-\alpha)^2}{\lambda -\alpha} < \frac{(2\alpha-\lambda) +\sqrt{(2\alpha -\lambda)^2  - 4(1-\alpha)^2}}{2}. $$
For future purpose we now define the following function $F_{2}:(2, +\infty)\times [0,1) \to \mathbb{R}$ given by
$$F_{2}(\lambda, \alpha):=-1+\alpha + \delta_{\alpha}- \theta'_{\alpha}=1 +2\alpha + \frac{(1-\alpha)^2}{\lambda -\alpha} - \frac{(2\alpha-\lambda) +\sqrt{(2\alpha -\lambda)^2  - 4(1-\alpha)^2}}{2}.$$
By definition, when $\alpha \to 0$ both $\alpha -\lambda$ and $(2\alpha -\lambda)^2  - 4(1-\alpha)^2$ are respectively non zero and positive, so $F_{2}$ is continuous as algebraic combination of simple functions.

Moreover, we have
$$\frac{\partial F_{2}}{\partial \lambda} = -\frac{\left(1-\alpha \right)^{2}}{\left(\lambda -\alpha \right)^{2}}+\frac{1}{2}+\frac{2 \alpha -\lambda}{2 \sqrt{\left(2 \alpha -\lambda \right)^{2}-4 \left(1-\alpha \right)^{2}}}$$
so $F_{2}$ is differentiable with respect to $\lambda$ in its domain.

\begin{figure}[H]
  \centering
  \includegraphics[width=10cm]{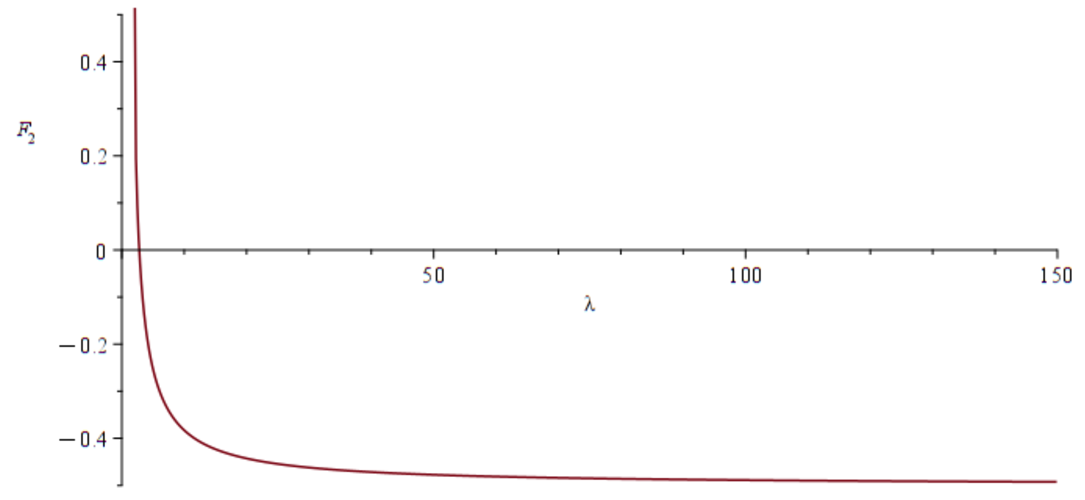}
  \caption{The graph of $F_{2}(\lambda, 1/4)$.}\label{fig:f2}
\end{figure}

We notice that the following chain of equivalences holds
$$\frac{1}{2}+\frac{2 \alpha -\lambda}{2 \sqrt{\left(2 \alpha -\lambda \right)^{2}-4 \left(1-\alpha \right)^{2}}}<0$$
%$$\frac{1}{2}- \frac{1}{2 \sqrt{1 -\frac{4 \left(1-\alpha \right)^{2}}{\left(2 \alpha -\lambda \right)^{2}}}}<0\quad (\times 2)$$
%$$ 1 - \frac{1}{\sqrt{1 -\frac{4 \left(1-\alpha \right)^{2}}{\left(2 \alpha -\lambda \right)^{2}}}}<0$$
%$$\sqrt{1 -\frac{4 \left(1-\alpha \right)^{2}}{\left(2 \alpha -\lambda \right)^{2}}} <1$$
%$$\sqrt{1 -\frac{4 \left(1-\alpha \right)^{2}}{\left(2 \alpha -\lambda \right)^{2}}} <1$$
$$\frac{4 \left(1-\alpha \right)^{2}}{\left(2 \alpha -\lambda \right)^{2}}>0,$$
which is always the case for $\alpha < 1$.

We notice that $\frac{\partial F_{2}}{\partial \lambda}<0$, $\displaystyle\lim_ {\lambda \to 2^{+}} F_{2}(\lambda, \alpha)= \alpha +\frac{\left(1-\alpha \right)^{2}}{2-\alpha}>0$ and $\displaystyle\lim_ {\lambda \to +\infty} F_{2}(\lambda, \alpha)= -1+2\alpha <0$ thus, there exists a unique  $\lambda:=\tau_{2}(\alpha)$ such that $ F_{2}(\lambda, \alpha)<0$ for all $\lambda >\tau_{2}(\alpha)$, see Figure~\ref{fig:f2}. Our reasoning only works for $\alpha <1/2$, because if $\alpha \to 1/2$ then $\displaystyle\lim_ {\lambda \to +\infty} F_{2}(\lambda, \alpha)= -1+2\alpha = 0$ thus $F_{2}(\lambda, \alpha)=0$ has no solution and $F_{2}(\lambda, \alpha)>0$ for all $\lambda>2$.

In conclusion, for each $0 \leq \alpha <1/2$ there will be a number $\tau_{2}(\alpha)$ such that $\displaystyle\lim_ {k\to \infty} \rho_{\alpha}(A_{\alpha}(G_k)) = \lambda,$ for $\lambda >  \tau_{2}(\alpha)$.  For $\lambda = \tau_{2}(\alpha)$ we can take a decreasing sequence $\lambda_{n} \to  \tau_{2}(\alpha)$ to prove that $\tau_{2}(\alpha)$ is an $A_{\alpha}$-limit point.
\end{proof}

As expected, for $\alpha=0$ we get this inequality true for $\lambda > \tau_{2}(0)= 1/6\,\sqrt [3]{108+12\,\sqrt {69}}+2\,{\frac {1}{\sqrt [3]{108+12\,
\sqrt {69}}}}+1 = 2.324717957$ as predicted from \cite{shearer1989distribution} that is, $\displaystyle\lim_ {k\to \infty} \rho_{\alpha}(A_{0}(G_k)) = \lambda,$ for $\lambda >  2.324717957$. In that work the proof is separated in two parts, the first one is $-1< b_k$ works only for $\lambda >  2.324+$.

By numerical inspection, the correspondence $\alpha \mapsto \tau_{2}(\alpha)$ appears to be increasing with respect to $\alpha$. It seems that $\tau_{2}(\alpha) \to \infty$ when $\alpha \to 1/2$.
Indeed $\tau_{2}(0.499999)= 2501.750025$. Not by coincidence, $Q(G)=2 A_{1/2}(G)$ where $Q$ is the signless Laplacian, leading to investigate the distribution of the Laplacian limit points, which is, as far as we can see, a very hard problem! It will be investigated in a future research.
For example, for $\alpha=1/4$(see Figure~\ref{fig:f2}) we see that $\tau_{2}(1/4)= 2.795171086$ that is, $\displaystyle\lim_ {k\to \infty} \rho_{\alpha}(A_{1/4}(G_k)) = \lambda,$ for $\lambda >  2.795171086$. Table~ \ref{table:thau_2-values} shows the behaviour of $\tau_{2}(\alpha)$.
\begin{table}[H]
  \centering
  \begin{tabular}{|c|r|}
   \hline
   $\alpha$ & $\tau_{2}(\alpha)$\\
   \hline
    0 & 2.324717958  \\
    $10^{-5}$ & 2.324726949  \\
    $10^{-4}$ & 2.324807890  \\
    $10^{-3}$ & 2.325619037  \\
    $10^{-2}$ & 2.333907609  \\
    $10^{-1}$ & 2.439018189  \\
    $0.4$ & 4.271267076  \\
    $0.49$ & 26.75245169   \\
    $0.499$ & 251.7502495   \\
   \hline
 \end{tabular}
  \caption{Values of $\tau_{2}(\alpha)$.}\label{table:thau_2-values}
\end{table}\begin{center}

\end{center}

\begin{remark}
  A closer look at Theorem~\ref{thm: main alpha limit one} shows that we can actually control the speed in which $\displaystyle\lim_ {k\to \infty} \rho_{\alpha}(A_{\alpha}(G_k)) = \lambda$. Indeed, in the proof of Theorem~\ref{thm: main alpha limit one} we obtain
$1 < \frac{(1-\alpha)^2}{b_{j}^2}$ for $\lambda > \tau_{2}(\alpha)$. In particular, from Theorem~\ref{thm: criteria two} we have $\varepsilon_{k}< \sigma_{k} = \frac{-b_k}{\frac{d\, b_k}{d\, \varepsilon}(0)}$ and, from the proof of Theorem~\ref{thm: criteria three}, we also know that
$$\frac{d\, b_k}{d\, \varepsilon}(0) >  1+  \frac{(1-\alpha)^2}{b_{k-1}^2}  +  \frac{(1-\alpha)^2}{b_{k-1}^2}\frac{(1-\alpha)^2}{b_{k-2}^2} + \cdots +  \frac{(1-\alpha)^2}{b_{k-1}^2}\cdots\frac{(1-\alpha)^2}{b_{1}^2}> k.$$
Thus
$$ \lambda - \rho_{\alpha}(A_{\alpha}(G_k)) < \varepsilon_{k} < \sigma_{k} = \frac{-b_k}{\frac{d\, b_k}{d\, \varepsilon}(0)}<\frac{-b_k}{k}< \frac{\delta_{\alpha}- \theta'_{\alpha}}{k}.$$
Since $C:=\delta_{\alpha}- \theta'_{\alpha}$ is a fixed number depending on $\lambda $ and $\alpha$ we obtain that the error $| \lambda - \rho_{\alpha}(A_{\alpha}(G_k))|$ decays as $C \frac{1}{k}$ for $\alpha < 1/2$ and $\lambda > \tau_{2}(\alpha)$.
\end{remark}

\begin{example}
   Consider $\lambda=2.44$ and $\alpha=0.1$. We notice that $\lambda > \tau_{2}(\alpha)= 2.439018189$ so Theorem~\ref{thm: main alpha limit one} holds, that is, if  $G_{k}(\lambda)=[r_1, r_2, \ldots, r_k]$ is the $\alpha$-Shearer sequence associated to $\lambda$ then $\displaystyle\lim_ {k\to \infty} \rho_{0.1}(A_{0.1}(G_k)) = 2.44.$\\
   As an illustration, we can check numerically the reasoning used at the proof. Taking $k=100$, we obtain\\
    $G_{k}(\lambda)=[4, 0, 1, 1, 1, 1, 0, 1, 1, 1, 1, 1, 1,1, 1, 1, 1, 1, 1, 1, 1, 0, 0, 1, 1, 1, 1, 1,1, 0, 1, 1, 1, 1, 1, 1,$\\ $ 1, 1, 1, 1, 1, 1, 1, 1, 0, 1, 1, 1, 1, 1, 1, 1, 1, 0, 1, 0, 1, 1, 1, 1, 1, 1, 1, 1,1, 1, 1, 1, 1, 1, 1,0, 0, 1, 1, 1,$\\ $ 1, 1, 1, 1, 1, 0, 1, 1, 1, 1, 1, 1, 1, 1, 1, 1, 0, 1, 1, 1, 0, 1, 0, 0]$.\\
   The diagonalization algorithm produces\\
   $\operatorname{Diag}(G_{100}(\lambda), -\lambda)=(b_1, \ldots, b_{100})= [-0.555, -0.782, -0.757, -0.724, -0.676, -0.595,$\\ $ -0.879, -0.873, -0.866, -0.858, -0.85, -0.841, -0.83, -0.818, -0.804, -0.786, ...,$\\
   %$-0.764, -0.733, -0.689, -0.619, -0.485, -0.569, -0.815, -0.8, -0.782, -0.758, -0.725, -0.677, -0.597, -0.883, -0.876, -0.87, -0.862, -0.855, -0.846, -0.837, -0.826, -0.813, -0.797, -0.778, -0.752, -0.717, -0.665, -0.575, -0.832, -0.821,$\\ $  -0.807, -0.79, -0.768, -0.74, -0.699, -0.635, -0.518, -0.677, -0.598, -0.886, -0.879, -0.873, -0.866, -0.858, -0.85, -0.841, -0.83, -0.818, $\\ $ -0.804, -0.787, -0.764, -0.734, -0.69, -0.62, -0.487, -0.577, -0.837, -0.826, -0.814, -0.798, -0.779, -0.755, -0.72, -0.669, -0.584, $\\ $ -0.852, -0.844, -0.834, -0.822, -0.809, -0.792, -0.772, -0.744, -0.705, -0.645, -0.539, -0.736, -0.693, $\\
    $ -0.625, -0.499, -0.616, -0.478, -0.546, -0.856]$.\\
   For $j=1,..., 100$ we have $-0.9=-1+\alpha< b_{j}$  as it is suppose to be.   By the way, a numerical computation shows that $\rho_{0.1}(A_{0.1}(G_{100})) \simeq  2.4399999999999995$.
\end{example}

\section{Small values of $\lambda$}\label{sec:small}

A more difficult problem is to consider $A_{\alpha}$-limit points for $\lambda < \tau_{2}(\alpha)$. We recall that the proof in \cite{shearer1989distribution} is for $\alpha=0$, and it is separated in two parts; the first one is when $-1< b_k$ and works only for $\lambda >  2.324+$. After that, the author analyses the interval $2.058+ < \lambda < 2.324+$ where one can have $b_k< -1$, but a rearrangement of the product will ensure the result. We now seek the analogous result for $\alpha$ close to 0.
In pursuing that, we obtain some new features of our approach. In particular, we find an upper bound $\tau_{1}'(\alpha)< \tau_{2}(\alpha)$  that may produce intervals of unknown behaviour.  More precisely, we will show that when $\alpha$ is small, there are numbers $\tau_{1}(\alpha)< \tau_{1}'(\alpha)< \tau_{2}(\alpha)$ such that any $\lambda \in (\tau_{1}(\alpha), \tau_{1}'(\alpha))$  is an $A_{\alpha}$-limit point. In some cases, this feature produces gaps in the density  of the $A_{\alpha}$-limit point. In other words, for some values of $\alpha$, there exist intervals for which we do not know whether their points are $A_{\alpha}$-limit points.

\begin{theorem}\label{thm: main alpha limit two} Let $\alpha^{*}:=\frac{3-\sqrt{2}}{7}= 0.226540919+$ and consider $\alpha \in [0,\alpha^{*})$.\\
Then any $\lambda \in [\tau_{1}(\alpha), \tau'_{1}(\alpha))$ is an $A_{\alpha}$-limit point,
where $\tau_{1}(\alpha)< \tau'_{1}(\alpha)$ are,\\ for $0<\alpha<\alpha^{*}$,  the  solutions of  $\left(2\alpha -\lambda +\delta_{\alpha}\right) \left(\theta'_{\alpha} - \delta_{\alpha} \right) =2(1-\alpha)^2$, and\\
for $\alpha=0$, $\tau_{1}(\alpha)=\sqrt{2+\sqrt{5}}$ and $ \tau'_{1}(\alpha)=\infty$.  \\

\noindent In particular the number $\alpha^{*}$ satisfies  $\tau_{1}(\alpha^{*})=\tau'_{1}(\alpha^{*})$ and is the largest positive number such that $\tau_{1}(\alpha)< \lambda < \tau'_{1}(\alpha)$ and $\left(2\alpha -\lambda +\delta_{\alpha}\right) \left(\theta'_{\alpha} - \delta_{\alpha} \right) - 2(1-\alpha)^2 < 0$ for $\alpha<\alpha^{*}$.
\end{theorem}

\begin{proof} We notice that the case $\alpha=0$ is given by Shearer's result. So we may assume that $\alpha>0$. We let $G_k$ be the $\alpha$-Shearer sequence associated with a given $\lambda>2$. We want to show that $\displaystyle\lim_ {k\to \infty} \rho_{\alpha}(A_{\alpha}(G_k)) = \lambda$. Consider $b_{j+1}= 2\alpha -\lambda - \frac{(1-\alpha)^2}{b_{j}} + r_{j+1} \delta_{\alpha}$ and suppose that $r_{j+1} \geq 1$, we want to show that $\frac{(1-\alpha)^2}{b_{j}^2}\frac{(1-\alpha)^2}{b_{j+1}^2}>1$,  equivalently $(1-\alpha)^2> b_{j}b_{j+1}$  or $b_{j}b_{j+1} - (1-\alpha)^2<0$.  We observe that if this is true, then the sum in Theorem~\ref{thm: criteria three} diverges, proving our result.

Since $\delta_{\alpha}= \alpha+ \frac{(1-\alpha)^2}{\lambda  -\alpha} >0$ we have
$$b_{j+1}\geq 2\alpha -\lambda - \frac{(1-\alpha)^2}{b_{j}} +\delta_{\alpha}$$
$$b_{j+1}b_{j} \leq \left(2\alpha -\lambda +\delta_{\alpha}\right) b_{j} -(1-\alpha)^2.$$
As $2\alpha -\lambda +\delta_{\alpha}<0$ and $\theta'_{\alpha} - \delta_{\alpha} < b_{j}$, we obtain
$$b_{j+1}b_{j} \leq \left(2\alpha -\lambda +\delta_{\alpha}\right) \left(\theta'_{\alpha} - \delta_{\alpha} \right) -(1-\alpha)^2$$
$$b_{j+1}b_{j} -(1-\alpha)^2 \leq \left(2\alpha -\lambda +\delta_{\alpha}\right) \left(\theta'_{\alpha} - \delta_{\alpha} \right) -2(1-\alpha)^2$$

We now define the function $F_{1}:(2, +\infty)\times [0,1) \to \mathbb{R}$ given by
$$F_{1}(\lambda, \alpha):=\left(2\alpha -\lambda +\delta_{\alpha}\right) \left(\theta'_{\alpha} - \delta_{\alpha} \right) -2(1-\alpha)^2.$$
By construction, $b_{j}b_{j+1} - (1-\alpha)^2<0$ is a consequence of $F_{1}(\lambda, \alpha)<0$. By definition, when $\alpha \to 0$ both $\alpha -\lambda$ and $(2\alpha -\lambda)^2  - 4(1-\alpha)^2$ are respectively non zero and positive, so $F_{1}$ is continuous as algebraic combination of simple functions.

\begin{figure}[H]
  \centering
  \includegraphics[width=7cm]{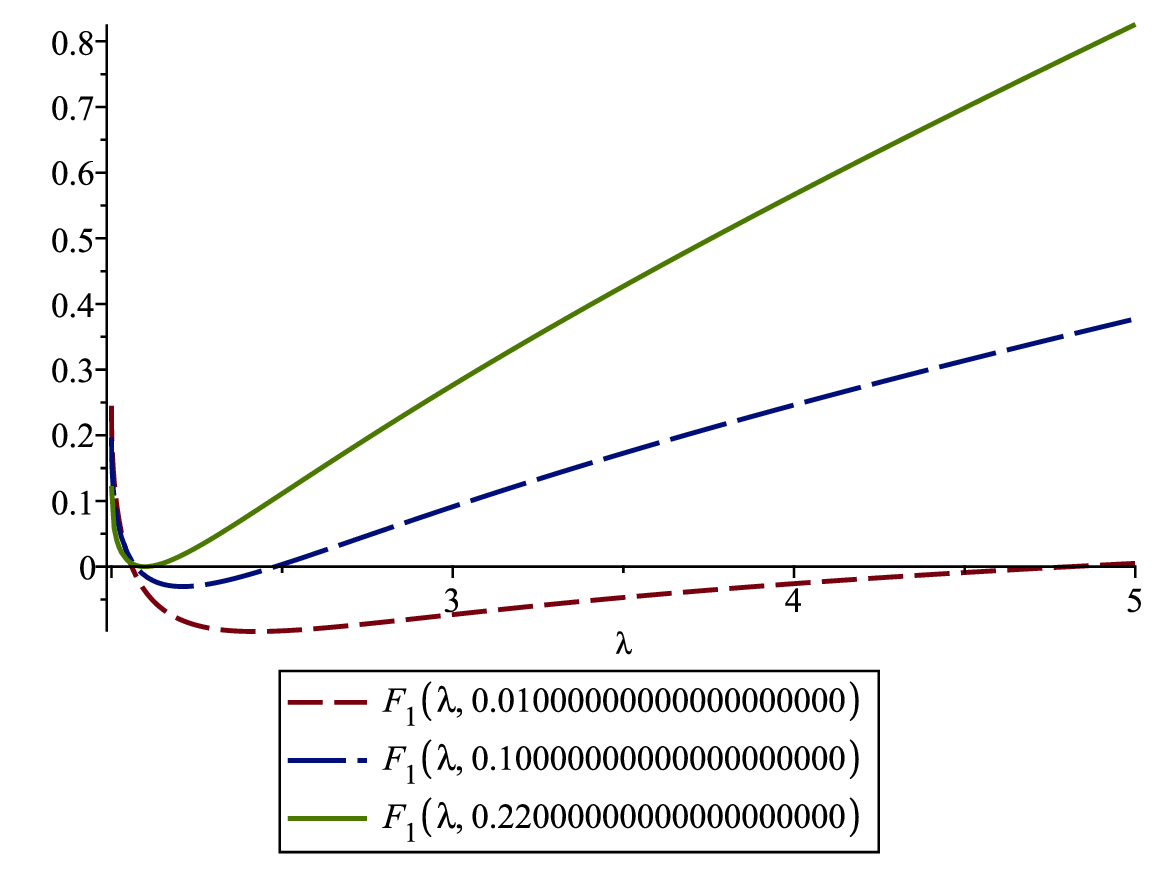}
  \caption{The graphs of $F_{1}(\lambda, 10^{-2}), F_{1}(\lambda, 10^{-1})$ and $F_{1}(\lambda, 0.22)$.}\label{fig:tau1}
\end{figure}

The inequality $F_{1}(\lambda, \alpha)<0$ implicitly defines $\lambda$ belonging to an interval which is a function of $\alpha$(see Figure~\ref{fig:tau1}). Let $(\tau_{1}(\alpha),\tau_{1}'(\alpha))$  be such function satisfying $F_{1}(\lambda, \alpha)<0$ for $\tau_{1}(\alpha)< \lambda < \tau_{1}'(\alpha)$ and $F_{1}(\tau_{1}(\alpha), \alpha)=0$ ($F_{1}(\tau'_{1}(\alpha), \alpha)=0$). Our proof that such numbers $\tau_{1}(\alpha)$ and $\tau_{1}'(\alpha)$ satisfying the above property actually exist, requires a series of technical (elementary) claims that are contained in Appendix~\ref{app:tau1}.

From the previous computations we see that while $r_{j+1}\neq 0$ we get  $ b_{j}b_{j+1}< (1-\alpha)^2$, guaranteeing our result. Otherwise, if $r_{j+1}=r_{j+2}=\cdots=r_{j+m}=0$ we are iterating the recursion
$$b_{j+1}= 2\alpha -\lambda - \frac{(1-\alpha)^2}{b_{j}}$$
which is the same  $Z_j$, that appear in  Theorem~\ref{thm: alpha stalike limit}. Hence, if $m$ is not finite we see that $b_{j+m}$ is monotonously decreasing to $\theta_{\alpha}$. Notice that
$\theta'_{\alpha} - \theta_{\alpha}> \delta_{\alpha}$ that is, $\theta_{\alpha} < \theta'_{\alpha} -  \delta_{\alpha}$, for any $\lambda > \tau_{1}(\alpha)$ (recall that $\delta_{\alpha} -(\theta'_{\alpha} - \theta_{\alpha}) = F_{0}(\lambda, \alpha)<0$ for $\lambda> \tau_{0}(\alpha)=\tau_{1}(\alpha)$ for $\alpha < \alpha^{*}$). Thus, if $b_{j+m}$ is close enough to  $\theta_{\alpha}$ then we have $\varphi_{\alpha}(b_{j+m}) < b_{j+m} < \theta'_{\alpha} - \delta_{\alpha}$ or $2\alpha -\lambda - \frac{(1-\alpha)^2}{b_{j+m}}+ \delta_{\alpha} < \theta'_{\alpha} $.
Since $b_{j+m+1}= 2\alpha -\lambda - \frac{(1-\alpha)^2}{b_{j+m}} + r_{j+m+1} \delta_{\alpha} < \theta'_{\alpha}$ we obtain $r_{j+m+1} \geq 1$. As a mater  of fact, such $m$ is globally bounded, but it depends on $\alpha$.

This means that in the $\alpha$-Shearer sequence of caterpillars $G_k$ associated with a given $\lambda > 2$, the number of consecutive vertices on the back nodes that have  no pendant  vertices is finite.

Consider then such a finite sequence with no pendant vertices, say $r_j=r_{j+1}=\ldots=r_{j+m}$ and $r_{j+m+1} \neq 0$. We have $b_{j+m}b_{j+m+1}< (1-\alpha)^2$ we recall that
$$b_{j+m}= 2\alpha -\lambda - \frac{(1-\alpha)^2}{b_{j+m-1}}$$
thus
$$b_{j+m-1} = \frac{(1-\alpha)^2}{( 2\alpha -\lambda ) - b_{j+m}}<0.$$
Also,
$$b_{j+m+2}= 2\alpha -\lambda - \frac{(1-\alpha)^2}{b_{j+m+1}} + r_{j+m+2} \delta_{\alpha}> 2\alpha -\lambda - \frac{(1-\alpha)^2}{b_{j+m+1}}.$$
Multiplying by the previous equation we get
$$b_{j+m-1}b_{j+m+2} < \left( 2\alpha -\lambda - \frac{(1-\alpha)^2}{b_{j+m+1}}  \right)\left( \frac{(1-\alpha)^2}{ ( 2\alpha -\lambda ) - b_{j+m}}\right)$$
%$$b_{j+m-1}b_{j+m+2} < \left(  \frac{(2\alpha -\lambda)b_{j+m+1} -(1-\alpha)^2}{b_{j+m+1}}  \right)\left( \frac{(1-\alpha)^2}{ ( 2\alpha -\lambda ) - b_{j+m}}\right)$$
%$$b_{j+m-1}b_{j+m+2} < \left(  \frac{(2\alpha -\lambda)b_{j+m+1} -(1-\alpha)^2}{b_{j+m+1}}  \right)\left( \frac{(1-\alpha)^2}{( 2\alpha -\lambda ) - b_{j+m}}\right)$$
$$b_{j+m-1}b_{j+m+2} <  \frac{(2\alpha -\lambda)b_{j+m+1} -(1-\alpha)^2}{( 2\alpha -\lambda )b_{j+m+1} - b_{j+m} b_{j+m+1}} \, (1-\alpha)^2.$$
As we already establish that  $b_{j+m} b_{j+m+1}< (1-\alpha)^2$ we get $\frac{(2\alpha -\lambda)b_{j+m+1} -(1-\alpha)^2}{( 2\alpha -\lambda )b_{j+m+1} - b_{j+m} b_{j+m+1}}<1$ thus $b_{j+m-1}b_{j+m+2} < (1-\alpha)^2$.

By repeating this same argument we conclude that $b_{j+m-2}b_{j+m+3} < (1-\alpha)^2$,\\ $b_{j+m-3}b_{j+m+4} < (1-\alpha)^2$, etc.

Consequently, the sum $1+  \frac{(1-\alpha)^2}{b_{k-1}^2}  +  \frac{(1-\alpha)^2}{b_{k-1}^2}\frac{(1-\alpha)^2}{b_{k-2}^2} + \cdots +  \frac{(1-\alpha)^2}{b_{k-1}^2}\cdots\frac{(1-\alpha)^2}{b_{1}^2}$ will blowup as $k \to \infty$ because wherever we get $\frac{(1-\alpha)^2}{b_{j}^2}<1$ we can match it with another  $j'$ such that $b_{j}b_{j'} < (1-\alpha)^2$ so that $\frac{(1-\alpha)^2}{b_{j}^2}\frac{(1-\alpha)^2}{b_{j'}^2}>1$.

According to Theorem~\ref{thm: criteria three}, it means that $\displaystyle\lim_ {k\to \infty} \rho_{\alpha}(A_{\alpha}(G_k)) = \lambda,$ for $\tau_{1}(\alpha) < \lambda < \tau_{1}'(\alpha)$.  To conclude the proof  we need to include $\lambda=\tau_{1}(\alpha)$. Notice that by Theorem~\ref{thm: alpha stalike limit} we get  $\displaystyle \lim_{n \to \infty}\rho_{\alpha}(A_{\alpha}(T_{1,n,n}))= \tau_{0}(\alpha)=\tau_{1}(\alpha)$, meaning that $\displaystyle\lim_ {k\to \infty} \rho_{\alpha}(A_{\alpha}(G_k)) = \lambda,$ for $\tau_{1}(\alpha) \leq \lambda < \tau_{1}'(\alpha)$.

The number $\alpha^{*}=\frac{3}{7}-\frac{\sqrt{2}}{7}$ is the maximum value for which there exists the open interval $(\tau_{1}(\alpha), \,\tau_{1}'(\alpha))$ for $\alpha<\alpha^{*}$.  Indeed, as we know, $\tau_{1}(\alpha)=\tau_{0}(\alpha)$ for small positive values of  $\alpha$. The limit point for this property is when $F_{1}$ has only one root denoted $\alpha^{*}$, which is necessarily  equal to $\tau_{0}(\alpha)$, that is, $\tau_{1}'(\alpha)=\tau_{0}(\alpha)$. As $F_{1}(\lambda, \alpha)=-F_{0}(\lambda, \alpha)\, F_{3}(\lambda, \alpha)$, the previous condition is equivalent to find $\alpha^{*} >0$ and $\lambda^{*}$ such that
\[\left\{
    \begin{array}{ll}
      F_{0}(\lambda^{*}, \alpha^{*})= \delta_{\alpha^{*}} - (\theta'_{\alpha^{*}} - \theta_{\alpha}) =0\\
      F_{3}(\lambda^{*}, \alpha^{*})=\delta_{\alpha^{*}} + \theta'_{\alpha^{*}}=0
    \end{array}
  \right.
\]
Subtracting the equations we get $2\theta'_{\alpha^{*}}= \theta_{\alpha^{*}}$, which is equivalent to
$$-\sqrt{\left(2 \alpha^{*} -\lambda^{*} \right)^{2}-4 \left(1-\alpha^{*} \right)^{2}}
 =
\alpha^{*} -\frac{\lambda^{*}}{2}+\frac{\sqrt{\left(2 \alpha^{*} -\lambda^{*} \right)^{2}-4 \left(1-\alpha^{*} \right)^{2}}}{2}.
$$
From this equation we obtain $\lambda^{*} = 2 \alpha^{*} +\frac{3 \sqrt{2}\, \alpha^{*}}{2}-\frac{3 \sqrt{2}}{2}$ or $\lambda^{*} = 2 \alpha^{*} -\frac{3 \sqrt{2}\, \alpha^{*}}{2}+\frac{3 \sqrt{2}}{2}$. By substituting this in the first equation we obtain four complex roots for the first equation and only one real root (plus two complex) for the second one, thus
$$\alpha^{*}=\frac{3}{7}-\frac{\sqrt{2}}{7} = 0.2265409196609... \text{ and }  \lambda^{*}=\frac{9}{7}+\frac{4 \sqrt{2}}{7}= 2.0938363213560...$$
We refer to Table~\ref{tab:tau1} for illustrating the values $\tau_{1}(\alpha)$ and $\tau_{1}'(\alpha)$, for a few values of $\alpha$.
\end{proof}
\begin{table}[H]
  \centering
 \begin{tabular}{|c|r|}
   \hline
   $\alpha$ & $(\tau_{1}(\alpha),\tau_{1}'(\alpha))$\quad\quad\quad $ $\\
   \hline
    0 &         $(2.058171027,   \infty)$ \quad\quad\quad\;\; $ $ \\
    $10^{-5}$ & $(2.058172154,  46.43683033)$  \\
    $10^{-4}$ & $(2.058182294, 21.58805390)$ \\
    $10^{-3}$ & $(2.058283826, 10.08827222)$  \\
    $10^{-2}$ & $(2.059312583, 4.810633985)$  \\
    $10^{-1}$ & $(2.071110742, 2.479706668)$  \\
    $0.22$ &   $(2.092435365, 2.103408681)$  \\
    $0.2265409$ &   $(2.093719372, 2.094603459)$  \\
   \hline
 \end{tabular}
  \caption{Values of the interval$ (\tau_{1}(\alpha),\tau_{1}'(\alpha))$}\label{tab:tau1}
\end{table}

\begin{remark}
  We would like to explain further details for the divergence reasoning in the proof of Theorem~\ref{thm: main alpha limit two}. The quantity
$$Q_k:=\frac{(1-\alpha)^2}{b_{k-1}^2}  +  \frac{(1-\alpha)^2}{b_{k-1}^2}\frac{(1-\alpha)^2}{b_{k-2}^2} + \cdots +  \frac{(1-\alpha)^2}{b_{k-1}^2}\cdots\frac{(1-\alpha)^2}{b_{1}^2}$$
will increase according to the number of summands that are bigger than 1. In the proof we have shown that the number $\frac{(1-\alpha)^2}{b_{j}^2}>1$ otherwise $\frac{(1-\alpha)^2}{b_{j-1}^2}\frac{(1-\alpha)^2}{b_{j}^2}>1$, provided $r_j \neq 0$. In  the case there are sequences of consecutive vertices of the caterpillar with no pendant vertices, we have a run $r_j=r_{j+1}=\ldots=r_{j+m}$ and $r_{j+m+1} \neq 0$ satisfying, $\frac{(1-\alpha)^2}{b_{j+m}^2}\frac{(1-\alpha)^2}{b_{j+m+1}^2}>1$, $\frac{(1-\alpha)^2}{b_{j+m-1}^2}\frac{(1-\alpha)^2}{b_{j+m+2}^2}>1$, and so on. Thus, fixed $k$, we see that $Q_k$ is greater than or equal to the number of cycles having the above property, where a cycle is an interval of integers $k-1, k-2,..., k-j$ such that $\frac{(1-\alpha)^2}{b_{k-1}^2}\cdots\frac{(1-\alpha)^2}{b_{k-j}^2}> 1$.  For example, suppose that  $\frac{(1-\alpha)^2}{b_{1}^2}>1$,  $\frac{(1-\alpha)^2}{b_{2}^2}>1$,  $\frac{(1-\alpha)^2}{b_{3}^2}>1$,   $\frac{(1-\alpha)^2}{b_{4}^2}<1$ but $\frac{(1-\alpha)^2}{b_{4}^2}\frac{(1-\alpha)^2}{b_{5}^2}>1$, $\frac{(1-\alpha)^2}{b_{6}^2}<1$ but $\frac{(1-\alpha)^2}{b_{6}^2}\frac{(1-\alpha)^2}{b_{15}^2}>1$, $\frac{(1-\alpha)^2}{b_{7}^2}\frac{(1-\alpha)^2}{b_{14}^2}>1$, $\frac{(1-\alpha)^2}{b_{8}^2}\frac{(1-\alpha)^2}{b_{13}^2}>1$,
$\frac{(1-\alpha)^2}{b_{9}^2}\frac{(1-\alpha)^2}{b_{12}^2}>1$, $\frac{(1-\alpha)^2}{b_{10}^2}\frac{(1-\alpha)^2}{b_{11}^2}>1$.\\
It is easy to see that we have five cycles $\{1\to 15, 2\to 15, 3\to 15, 4\to 15, 6\to 15\}$ and,
$$Q_{16}\geq \frac{(1-\alpha)^2}{b_{15}^2}\cdots\frac{(1-\alpha)^2}{b_{6}^2} + \frac{(1-\alpha)^2}{b_{15}^2}\cdots\frac{(1-\alpha)^2}{b_{4}^2} +\frac{(1-\alpha)^2}{b_{15}^2}\cdots\frac{(1-\alpha)^2}{b_{3}^2}+$$
$$+ \frac{(1-\alpha)^2}{b_{15}^2}\cdots\frac{(1-\alpha)^2}{b_{2}^2} + \frac{(1-\alpha)^2}{b_{15}^2}\cdots\frac{(1-\alpha)^2}{b_{1}^2} > 5.$$
Since the  number of cycles is unbounded as $k$ increases, so is $Q_k$.
\end{remark}

\begin{example}
   Consider $\lambda=2.06$ and $\alpha=0.01$. We notice that $\lambda \in (\tau_{1}(\alpha),\tau_{1}'(\alpha))= (2.059312583, 4.810633985)$ so Theorem~\ref{thm: main alpha limit two} holds, that is, if  $G_{k}(\lambda)=[r_1, r_2, \ldots, r_k]$ is the $\alpha$-Shearer sequence associated to $\lambda$, then $\displaystyle\lim_ {k\to \infty} \rho_{0.01}(A_{0.01}(G_k)) = 2.06.$\\
   As an illustration, we can check numerically the reasoning used at the proof. Taking $k=100$, we get\\
    $G_{k}(\lambda)=[2, 0, 0, 0, 0, 0, 0, 0, 0, 0, 1, 0, 0, 0, 0, 0, 0, 0, 0, 0, 0, 0, 0, 0, 0, 0, 0, 0, 0, 0, 0, 0, 0, 0, 1,$\\ $0, 0, 0, 0, 0, 0, 0, 0, 0, 0, 0, 0, 0, 0, 0, 0, 0, 0, 0, 0, 0, 0, 0, 0, 1, 0, 0, 0, 0, 0, 0, 0, 0, 0, 0, 0, 0, 0, 0, 0,$\\ $ 0, 0, 0, 0, 0, 0, 0, 0, 1, 0, 0, 0, 0, 0, 0, 0, 0, 0, 0, 0, 0, 0, 0, 0, 0]$.\\
   The diagonalization algorithm produces\\
   $\operatorname{Diag}(G_{100}(\lambda), -\lambda)=(b_1, \ldots, b_{100})= [-1.074, -1.127, -1.171, -1.203, -1.225, -1.24,$\\
   $ -1.25, -1.256, -1.259, -1.262, -0.775, -0.776, -0.776, -0.778,-0.78,-0.783, -0.788, $\\
   $-0.796, -0.808, -0.828, -0.856, -0.895, -0.945, -1.003, -1.063, -1.118, -1.163, \ldots, $\\
   %-1.197, -1.221,-1.221,-1.238, -1.248, -1.255, -1.259, -1.261, -0.775, -0.775, -0.776, -0.777, -0.778, -0.78, -0.784, -0.789, -0.798, -0.812, -0.833, -0.863, -0.905, -0.957, -1.015, -1.075, -1.128, -1.171, -1.203, -1.225, -1.24, -1.25, -1.256, -1.259, -1.262, -0.775, -0.776, -0.776, -0.778, -0.78, -0.783, -0.788, -0.796, -0.809, -0.829, -0.858, -0.897, -0.948, -1.006, -1.065, -1.12, -1.165, -1.199, -1.222, -1.238, -1.248, -1.255, -1.259, -1.262, -0.775, -0.775, -0.776, -0.777, -0.778, -0.781,
   $ -0.785, -0.791, -0.801, -0.816, -0.839, -0.872, -0.916, -0.97, -1.03, -1.088, -1.15]$.\\
   For $j=1$ we do not have $-0.99=-1+\alpha< b_{1}=-1.0738048780487808$ and $r_2=0$ so we do not have $ b_{1}b_{2}< (1-\alpha)^2$. The same goes to $r_3,\ldots, r_{10}$ but, as predicted we find $r_{11}\neq 0$ thus
   \begin{center}
   $ b_{10}b_{11}= 0.9780973959081004 < (1-\alpha)^2=0.9801$,\\
   $ b_{9}b_{12}= 0.9768462311806901 < (1-\alpha)^2=0.9801$,\\
   $\cdots$\\
   $ b_{1}b_{20}= 0.8888252835590791 < (1-\alpha)^2=0.9801$.
   \end{center}
   Looking to the next index we get $-0.99=-1+\alpha< b_{21}=-0.8559245912071809$, and the same is true until we reach $b_{24}=-1.0026610413051416$ then we repeat the previous analysis, because $r[25]=0$, until we reach $r[35] \neq 0$.\\
   By the way, a numerical computation shows that $$\rho_{0.01}(A_{0.01}(G_{100})) \simeq 2.059998455508993.$$
\end{example}

Now we can combine the previous results to obtain a neighborhood of $\alpha=0$ where the intervals $[\tau_{0}(\alpha), \infty)$ are entirely filled by $A_{\alpha}$-limit points.
\begin{corollary}\label{cor: intervals}
Let $0 \leq \alpha <1-\frac{2 \sqrt{5}}{5}= 0.105572809+$.
Then $\lambda$ is an $A_{\alpha}$-limit point for any $\lambda \geq \tau_{0}(\alpha)$.
\end{corollary}
\begin{proof}
   The proof is a consequence of Theorem~\ref{thm: main alpha limit one} and Theorem~\ref{thm: main alpha limit two} which claims that the intervals $[\tau_{1}(\alpha),\tau_{1}'(\alpha))$ (recall that $\tau_{0}(\alpha)= \tau_{1}(\alpha)$) and $[\tau_{2}(\alpha),\infty)$ are formed by $A_{\alpha}$-limit points for $0\leq \alpha < \alpha^* =0.226+$ and $0\leq \alpha < 1/2$, respectively (see Figure~\ref{fig:intervals}).
   Since $\tau_{2}$ increases to $\infty$ when $\alpha \to 1/2$ the overlap between these two intervals will occurs until
   $$\tau_{1}'(\alpha)=\lambda= \tau_{2}(\alpha).$$
   Thus we must solve the system
   \[\left\{
    \begin{array}{ll}
      F_{2}(\lambda, \alpha)= 0\\
      F_{3}(\lambda, \alpha)=0
    \end{array}
  \right.
\]
We will omit the computations which are very similar to the ones we performed earlier, producing
$$-1+3 \alpha +\frac{2 \left(1-\alpha \right)^{2}}{\lambda -\alpha} = 0.$$
From where we isolate
$$\lambda= \frac{\alpha^{2}+3 \alpha -2}{-1+3 \alpha}.$$
Substituting that in the first equation we get a single real solution $\alpha=1-\frac{2 \sqrt{5}}{5}= 0.105572809+$ which corresponds to $\lambda=\frac{-7+5 \sqrt{5}}{3 \sqrt{5}-5}= 2.4472135954+$. Those threshold points are depicted in the Figure~\ref{fig:intervals} along with all the information we have from our main results.
\end{proof}
\begin{figure}[H]
  \centering
  \includegraphics[width=18cm]{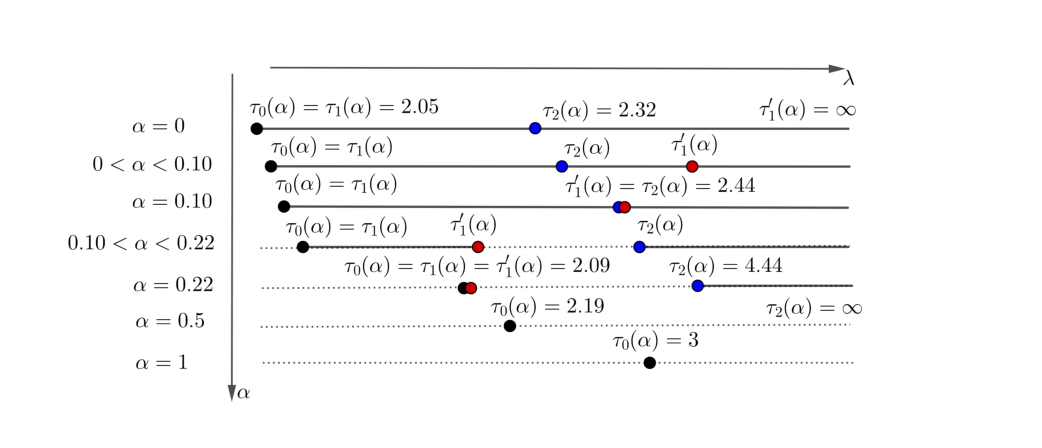}
  \caption{Summary of all information gathered from our results. The solid lines are intervals formed by $A_\alpha$-limit points.}\label{fig:intervals}
\end{figure}

\section{Concluding remarks}

We observe that Hoffman's original result \cite{hoffman1972limit} determine all the limit points of the adjacency matrix for $\lambda \in [2,\sqrt{2+\sqrt{5}})$. In fact, this is an enumerable set, in contrast with the extension given by Shearer that any $\lambda\geq \sqrt{2+\sqrt{5}}$ is a limit point. It would be interesting to study the $A_\alpha$ version of this result. A research problem is to characterize for small value $\alpha>0$, all $A_{\alpha}$-limit points smaller than $\tau_0(\alpha)=\tau_1(\alpha)$.

Finally, when $\tau_1'(\alpha)< \tau_2(\alpha)$ what happens in such gap? Notice that, for example, $\tau_1'(0.22)=2.103408681 < 2.692120306=\tau_2(0.22)$. The conclusion from  Theorem~\ref{thm: main alpha limit one} could hold in this interval because the proof uses only a sufficient condition.

\appendix
\section{The function $F_1$} \label{app:tau1}
Here we give a formal proof that $F_{1}(\lambda, \alpha):=\left(2\alpha -\lambda +\delta_{\alpha}\right) \left(\theta'_{\alpha} - \delta_{\alpha} \right) -2(1-\alpha)^2$ has two distinct roots $2< \tau_1(\alpha) < \tau'_1(\alpha)$ and $F_{1}(\lambda, \alpha)<0$ for $\tau_1(\alpha) < \lambda < \tau_1(\alpha)$ provided $\alpha \in [0,\alpha^*)$ where $\alpha^*=\frac{3}{7}-\frac{\sqrt{2}}{7}$.

\noindent \textbf{Claim 1:} $F_{1}(\lambda, \alpha)=-F_{0}(\lambda, \alpha)\, F_{3}(\lambda, \alpha)$ where $F_{0}(\lambda, \alpha) = \delta_{\alpha} - (\theta'_{\alpha} - \theta_{\alpha}) $ and $F_{3}(\lambda, \alpha) = \delta_{\alpha} + \theta'_{\alpha}$, where $F_{0}(\lambda, \alpha) $ is the same function of Theorem~\ref{thm: alpha stalike limit} defined by Equation~\eqref{eq:F0}. To see that we just use the already known relations such as $ 2\alpha -\lambda = \theta'_{\alpha} + \theta_{\alpha}$ and $\theta'_{\alpha}  \theta_{\alpha}= (1-\alpha)^2$.
Recall that
$$F_{1}(\lambda, \alpha)=\left(2\alpha -\lambda +\delta_{\alpha}\right) \left(\theta'_{\alpha} - \delta_{\alpha} \right) -2(1-\alpha)^2$$
and $\delta_{\alpha}=F_{0}(\lambda, \alpha)  + (\theta'_{\alpha} - \theta_{\alpha})$, thus
$$F_{1}(\lambda, \alpha)=\left(\theta'_{\alpha} + \theta_{\alpha}  +F_{0}(\lambda, \alpha)  + (\theta'_{\alpha} - \theta_{\alpha}) \right) \left(\theta'_{\alpha} - F_{0}(\lambda, \alpha)  - (\theta'_{\alpha} - \theta_{\alpha}) \right) -2(1-\alpha)^2=$$
$$=(F_{0}(\lambda, \alpha)  + 2 \theta'_{\alpha})(- F_{0}(\lambda, \alpha)  + \theta_{\alpha})-2(1-\alpha)^2=$$
$$=-F_{0}(\lambda, \alpha)^2 + F_{0}(\lambda, \alpha) \theta_{\alpha} - 2 F_{0}(\lambda, \alpha) \theta'_{\alpha} +2 \theta'_{\alpha} \theta_{\alpha} -2(1-\alpha)^2=$$ $$= -F_{0}(\lambda, \alpha) \left(F_{0}(\lambda, \alpha)- \theta_{\alpha} + 2 \theta'_{\alpha} \right)=-F_{0}(\lambda, \alpha) \left(\delta_{\alpha} - (\theta'_{\alpha} - \theta_{\alpha})- \theta_{\alpha} + 2 \theta'_{\alpha} \right)=$$
$$= -F_{0}(\lambda, \alpha) \left( \delta_{\alpha} + \theta'_{\alpha}\right)=-F_{0}(\lambda, \alpha)\, F_{3}(\lambda, \alpha).$$

\begin{figure}[H]
  \centering
  \includegraphics[width=7cm]{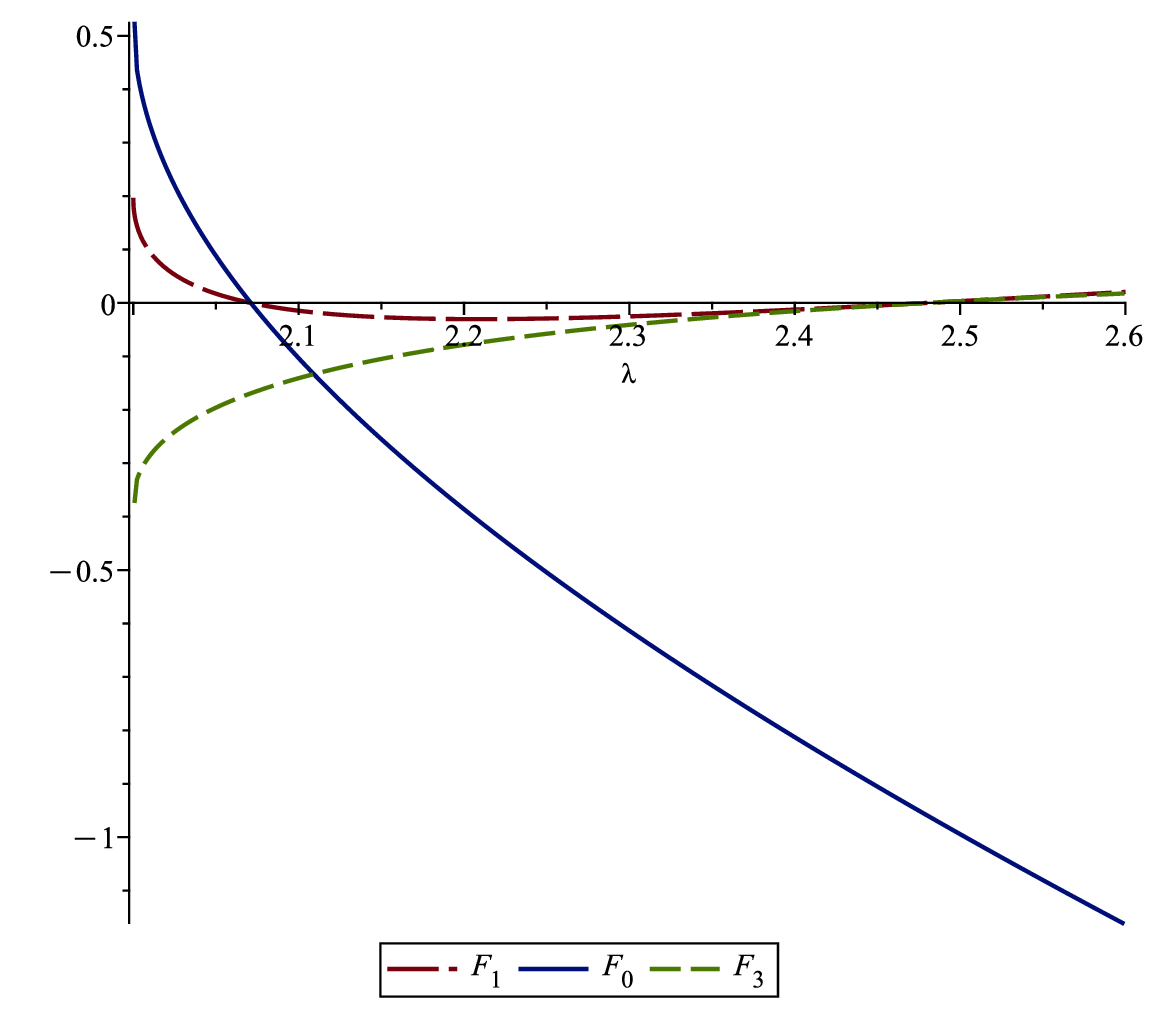}
  \caption{The graph of $F_{1}(\lambda, 0.1)$, $F_{0}(\lambda, 0.1)$ and $F_{3}(\lambda, 0.1)$. }\label{fig:tau0tau1}
\end{figure}
\vskip 1cm

\noindent \textbf{Claim 2:}  $F_{1}(\tau_{0}(\alpha), \alpha)=0$ (see from Figure~\ref{fig:tau0tau1} that both functions share a common root). In particular $\tau_{0}(\alpha)=\tau_{1}(\alpha)$ or $\tau_{0}(\alpha)=\tau_{1}'(\alpha)$.  It is evident from Claim 1, because $F_{0}(\tau_{0}(\alpha), \alpha)=0$, by definition.
\vskip 1cm

\noindent \textbf{Claim 3:} If $\alpha<1/4$ , then there exists an unique $\tau_{3}(\alpha)>2$ such that $F_{3}(\tau_{3}(\alpha), \alpha)=0$.\\
The proof is based in known properties of degree three polynomials. The equation $F_{3}(\lambda, \alpha)=0$ is equivalent to
$$2 \alpha +\frac{\left(1-\alpha \right)^{2}}{\lambda -\alpha}-\frac{\lambda}{2}+\frac{\sqrt{\left(2 \alpha -\lambda \right)^{2}-4 \left(1-\alpha \right)^{2}}}{2}=0,$$
which is equivalent to $\lambda$ be a real root of the polynomial $P_3(\lambda)=\lambda^3 + 5 \alpha \lambda^2 - (- 4\alpha ^ 2 - 6\alpha + 3) \lambda - \alpha^3 - 2\alpha^2 -3 \alpha +4 -\alpha^{-1}$. Performing a change of variables $\lambda (x):= x +\frac{5\alpha}{3}$ we obtain $ Q_3(x)= x^3 +(-13/3\,{\alpha}^{2}+6\,\alpha-3)x+(-{\frac {97\,{\alpha}^{3}}{27}}+8\,{\alpha}^{2}-8\,\alpha+4-{\alpha}^{
-1})$, whose roots are just a horizontal shift of those of $P_3$. Meaning that $P_3(\lambda)$ has a single real root if and only if $Q_3(\lambda)$ has a single real root. We know that a polynomial of degree three $x^3 + a x + b$ has only one real root if and only if $d:=-(4a^3+27b^2)<0$ where $a:=-13/3\,{\alpha}^{2}+6\,\alpha-3$ and $b:=-{\frac {97\,{\alpha}^{3}}{27}}+8\,{\alpha}^{2}-8\,\alpha+4-{\alpha}^{
-1}$, so we want to show
$$d(\alpha)= -23 \alpha^{6}+200 \alpha^{5}-732 \alpha^{4}+1496 \alpha^{3}-1886 \alpha^{2}+1512 \alpha -756+\frac{216}{\alpha}-\frac{27}{\alpha^{2}} <0,$$ for $0\leq\alpha<1/4$.\\
It will be useful to define some auxiliary quantities
$$n(\alpha):= -23 \alpha^{6}-732 \alpha^{4}-1886 \alpha^{2} \text{ and } p(\alpha):= 200 \alpha^{5}+1496 \alpha^{3}+1512 \alpha.$$
Also, define the increasing function $g(\alpha):= \frac{216}{\alpha}-\frac{27}{\alpha^{2}}$. By construction, we have $$d(\alpha)=n(\alpha)+p(\alpha)-756+g(\alpha).$$
Let us divide our proof in two parts. \\
(a) If $\alpha \leq 1/8$:\\

It is easy to see that, $n(\alpha)<0$ (decreasing), $p(\alpha)>0$ (increasing), $g(1/8)=0$ and $g(\alpha)<0$  for $\alpha < 1/8$. since $d(\alpha)=n(\alpha)+p(\alpha)-756+g(\alpha)$, we have  $$d(\alpha)=n(\alpha)+p(\alpha)-756 + g(\alpha)< 0 +p(1/8)-756 + 0= {\frac{786137}{4096}} -756 =-{\frac{2310439}{4096}}<0.$$

\noindent (b)$1/8<\alpha<1/4$:\\
We claim that $d'(\alpha)>0$, thus $d(\alpha)<d(1/4)=-176823/4096<0$. To prove that we compute $d'(\alpha)= n'(\alpha)+p'(\alpha)-756' + g'(\alpha)= (-138 \alpha^{5}-2928 \alpha^{3}-3772 \alpha) + (1000 \alpha^{4}+4488 \alpha^{2}+1512) + \left(-\frac{216}{\alpha^{2}}+\frac{54}{\alpha^{3}}\right)$.
Note that $-\frac{216}{\alpha^{2}}+\frac{54}{\alpha^{3}}= \alpha^{-3} ( 54 -216 \alpha)$ and $( 54 -216 \alpha)> 0$ for $1/8<\alpha<1/4$. For the first part we rewrite $(-138 \alpha^{5}-2928 \alpha^{3}-3772 \alpha) + (1000 \alpha^{4}+4488 \alpha^{2}+1512)$ as
$$\left(1000-138 \alpha \right) \alpha^{4}+\left(4488-2928 \alpha \right) \alpha^{2}+ (1512-3772 \alpha)> $$ $$>\left(1000-138(1/4) \right) \alpha^{4}+\left(4488-2928 (1/4) \right) \alpha^{2}+ (1512-3772 (1/4))=$$ $$= \frac{1931}{2} \alpha^{4}+3756 \alpha^{2}+569 >0.$$
\vskip 1cm

\noindent \textbf{Claim 4:} $F_{3}(\tau_{0}(\alpha), \alpha)<0$ and $\displaystyle \lim_{\lambda \to \infty}F_{3}(\lambda, \alpha)= \alpha >0$.\\
To prove that we start with the hardest part, $F_{3}(\tau_{0}(\alpha), \alpha)<0$. For $\lambda=\tau_{0}(\alpha)$ we have $F_{0}(\lambda, \alpha) =0$ so that
$ \delta_{\alpha} - (\theta'_{\alpha} - \theta_{\alpha}) =0$ and we would like to have $\delta_{\alpha} + \theta'_{\alpha} <0.$
Substituting the first one in the second expression, we obtain $2  \theta'_{\alpha}  <\theta_{\alpha}$, or
$$\left(2 \alpha -\lambda \right) + \sqrt{\left(2 \alpha -\lambda \right)^{2}-4 \left(1-\alpha \right)^{2}} > \frac{\left(2 \alpha -\lambda \right) - \sqrt{\left(2 \alpha -\lambda \right)^{2}-4 \left(1-\alpha \right)^{2}}}{2}.$$
A tedious computation shows that this inequality is equivalent to
$\lambda <\frac{3 \sqrt{2}\,}{2}(1-\alpha)+ 2 \alpha.$ Thus, we need to show that
$\tau_{0}(\alpha)< \frac{3 \sqrt{2}\,}{2}(1-\alpha)+ 2 \alpha.$ Since $\tau_{0}$ is an increasing function of $\alpha$ we obtain $\tau_{0}(\alpha)< \tau_{0}(\alpha^*)= \frac{9}{7}+\frac{4 \sqrt{2}}{7}$.   The function $\frac{3 \sqrt{2}\,}{2}(1-\alpha)+ 2 \alpha$ is decreasing with $\alpha$ so it attains its minimum for the interval $[0,\alpha^*]$ at $\alpha^*$ with the same value  $\frac{9}{7}+\frac{4 \sqrt{2}}{7}$.  Thus,
$$\tau_{0}(\alpha)< \tau_{0}(\alpha^*) \leq \frac{3 \sqrt{2}\,}{2}(1-\alpha)+ 2 \alpha,$$
for $\alpha \in [0,\alpha^*)$.\\
To see that $\displaystyle \lim_{\lambda \to \infty}F_{3}(\lambda, \alpha)= \alpha >0$ we just rewrite the formula
$$F_{3}(\lambda, \alpha)=  \alpha +\frac{\left(1-\alpha \right)^{2}}{\lambda -\alpha} +\frac{\left(2 \alpha -\lambda \right) + \sqrt{\left(2 \alpha -\lambda \right)^{2}-4 \left(1-\alpha \right)^{2}}}{2}=$$
$$=  \alpha +\frac{\left(1-\alpha \right)^{2}}{\lambda -\alpha} +
\frac{2 \left(1-\alpha \right)^{2}}{ \left(\left(2 \alpha -\lambda \right) - \sqrt{\left(2 \alpha -\lambda \right)^{2}-4 \left(1-\alpha \right)^{2}}\right)},$$
and, except by the first summand, the others vanishes when $\lambda \to \infty$.
\vskip 1cm

\noindent \textbf{Claim 5:} In the previous conditions, $\tau_{3}(\alpha)> \tau_{0}(\alpha)$ is the second root of $F_{1}$, meaning that $\tau_{1}(\alpha)=\tau_{0}(\alpha)< \tau_{3}(\alpha)=\tau_{1}'(\alpha)$ and $F_{1}(\lambda, \alpha)<0$ for $\tau_{1}(\alpha)< \lambda < \tau_{1}'(\alpha)$.\\
Indeed, $F_{3}(\tau_{3}(\alpha), \alpha)=0$ means that $F_{1}(\tau_{3}(\alpha), \alpha)=0$. Suppose, by contradiction, that $\tau_{0}(\alpha)> \tau_{3}(\alpha)$. As $F_{3}(\tau_{0}(\alpha), \alpha)<0$ and $\displaystyle \lim_{\lambda \to \infty}F_{3}(\lambda, \alpha)= \alpha >0$ the continuity of $F_{3}$ ensures that we can find a new real root for $F_{3}$, different from $\tau_{3}(\alpha)$, a contradiction with the uniqueness of $\tau_{3}(\alpha)$.  So far we have concluded that $\tau_{1}(\alpha)=\tau_{0}(\alpha)< \tau_{3}(\alpha)=\tau_{1}'(\alpha)$ are the only two consecutive roots of $F_1$. In addition, we recall that $-F_0$ is negative for $\lambda < \tau_{1}(\alpha)$ and  positive for $\lambda > \tau_{1}(\alpha)$ and, by the uniqueness of $\tau_{3}(\alpha)$, we note that $F_3$ is negative for $\lambda < \tau_{1}'(\alpha)$ and  positive for $\lambda >  \tau_{1}'(\alpha)$. Thus, $F_{1}(\lambda, \alpha)=-F_{0}(\lambda, \alpha)\, F_{3}(\lambda, \alpha)<0$ for $\tau_{1}(\alpha)< \lambda < \tau_{1}'(\alpha)$.
\vskip 1cm

\section*{Acknowledgments}
This work is partially supported CNPq under grant 408180/2023-4 - Chamada MCTI Nº 10/2023 - UNIVERSAL. V. Trevisan also acknowledges partial support of MATH-AMSUD under project GSA, brazilian team financed by CAPES - project 88881.694479/2022-01, and CNPq grant 310827/2020-5 .

\end{document}